\documentclass[12pt,a4paper]{amsart}
\usepackage[OT1]{fontenc}
\usepackage[latin1]{inputenc}
\usepackage[english]{babel}
\usepackage{amsfonts,amsmath,relsize,amsthm,units,amssymb,bbm,newtxtext,newtxmath}
\usepackage[footskip=.5in]{geometry}
\usepackage{xcolor,enumitem}
\usepackage[colorlinks]{hyperref}

\newtheorem*{theorem*}{Theorem}
\newtheorem{teo}{Theorem}[section]
\theoremstyle{definition}
\newtheorem{prop}[teo]{Proposition}
\newtheorem{lem}[teo]{Lemma}
\newtheorem{coro}[teo]{Corollary}
\newtheorem{defi}[teo]{Definition}
\newtheorem{rem}[teo]{Remark}
\newtheorem{ejem}[teo]{Example}
\newtheorem{problem}[teo]{Problem}

\DeclareMathOperator\ad{\mathrm{ad}}
\DeclareMathOperator\co{\mathrm{conv}}
\DeclareMathOperator\Ad{\mathrm{Ad}}

\DeclareMathOperator\ran{\mathrm{Ran}}

\DeclareMathOperator\diag{\mathrm{diag}}

\DeclareMathOperator\tr{\mathrm{Tr}}

\DeclareMathOperator\lu{\mathfrak{u}_n}
\DeclareMathOperator\slu{\mathfrak{su}_n}
\DeclareMathOperator\U{\mathsf{\mathbf U}_n}

\DeclareMathOperator\OO{\mathscr{O}}

\newcommand{\R}{\mathbb{R}}
\newcommand{\C}{\mathbb{C}}

\newcommand{\MM}{\mathfrak{M}}

\begin{document}

\makeatletter

\title[Sphere of a weakly invariant norm]{Weakly invariant norms: geometry of spheres in the space of skew-Hermitian matrices}
\author{Gabriel Larotonda}
\address{Departamento de Matem\'atica, FCEyN-UBA, and Instituto Argentino de Matem\'atica, CONICET, Buenos Aires, Argentina}
\email{glaroton@dm.uba.ar}
\author{Ivan Rey}
\address{Departamento de Matem\'atica, FCEyN-UBA, and Instituto Argentino de Matem\'atica, CONICET, Buenos Aires, Argentina}
\email{ivanrey1988@gmail.com}
\date{\today}
\makeatother


\begin{abstract}
Let $N$ be a weakly unitarily invariant norm (i.e. invariant for the coadjoint action of the unitary group) in the space of skew-Hermitian matrices $\mathfrak{u}_n(\mathbb C)$. In this paper we study the geometry of the unit sphere of such a norm, and we show how its geometric properties are encoded by the majorization properties of the eigenvalues of the matrices. We give a detailed characterization of norming functionals of elements for a given norm, and we then prove a sharp criterion for the commutator $[X,[X,V]]$ to be in the hyperplane that supports $V$ in the unit sphere. We show that the adjoint action $V\mapsto V+[X,V]$ of $\mathfrak{u}_n(\mathbb C)$ on itself pushes vectors away from the unit sphere. As an application of the previous results, for a strictly convex norm, we prove that the norm is preserved by this last action if and only if $X$ commutes with $V$. We give a more detailed description in the case of any weakly $Ad$-invariant norm.
\end{abstract}


\setlength{\parindent}{0cm} 

\maketitle

\tableofcontents

\section{Introduction}

The theory of unitarily invariant norms in the space $M_n(\mathbb C)$ of $n\times n$ complex matrices is well-established, and there is a huge amount of research around it, both theoretical and numerical. It has many applications outside matrix analysis, such as quantum information theory (see \cite{haya,hole,nielsen} and the references therein), and signal processing -via frames and convex optimization- \cite{absil,befi,fimi,tropp} to mention the most relevant. The theory of \textit{weakly} invariant norms however, is less popular, and since such a norm is defined by being invariant for the coadjoint action of the unitary group $\U=\U(\mathbb C)$, we propose that its natural setting is the Lie algebra of that group, the space  $\lu=\mathfrak u_n(\mathbb C)$ of skew-adjoint matrices. Moreover, since we have a correspondence with norms and unit balls, which are convex sets, we propose that the condition for the norm to be fully homogeneous (equivalently, that the unit ball is a balanced set) should be \textit{a priori} dropped. This leaves us with a larger and generous family of positively homogeneous and subadditive norms, which we call $\Ad$-invariant Finsler norms in $\lu$. This family of norms is in good correspondence with symmetric convex bodies of $\mathbb R^n$, and the study of the faces of the unit sphere can be engaged by majorization and convex analysis tools.

\medskip

It is noticeable that the (strong unitarily invariance) condition $\|UXV\|_{\phi}=\|X\|_{\phi}$ for unitary matrices $U,V$ and $X\in M_n(\mathbb C)$ ensures the correspondence between (strongly) unitarily invariant norms and symmetric gauge functions in $\mathbb R^n$. Moreover, there is a finite family of norms (the Ky-Fan norms) that control the size of two given matrices:  $\|X\|_{\phi}\le \|Y\|_{\phi}$ for any unitarily invariant norm if and only if the inequality holds for the Ky-Fan norms $\|X\|_{(k)}\le \|Y\|_{(k)}$, $k=1,\dots,n$. On the other hand, it is well-known that it does not exist such finite family that controls the size of $X,Y$ for any given weakly invariant  norm.  We have here \textit{orbit norms} as a good substitute for this control; these are the $\Ad$-invariant Finsler norms whose dual unit ball is exactly $\co(\OO_C)$, the convex hull of the unitary orbit of a  zero trace matrix $C\in \lu$. These orbit norms are a simplified (and adapted to the matrix setting) version of the Hofer norms  introduced in \cite{larmi}. The norming functionals of a given norm can then be completely described, and the same happens with the faces of the unit sphere (which is in several important cases a polytope) and its extreme points. These descriptions are then generalized to obtain properties of norming functionals of any Finsler norm (Proposition \ref{ordenados}), which we then apply to the main result of this paper (Theorem \ref{teoN}): we characterize those $X$ such that $V+[X,[X,V]]$ belongs to the affine hyperplane supporting $V$ in the sphere of the norm (extending a result which is apparent for the Frobenius norm).  With  similar proofs, we note that $\|V+[X,V]\|\ge \|V\|$ for any $X\in\lu$ and any $\Ad$-invariant Finsler norm (this can be read in terms of Birkhoff orthogonality, see Remark \ref{bhort}): in particular if $[X,V]\ne 0$, then the inequality is strict for any strictly convex norm. It is worth noticing that this notion of Birkhoff orthogonality (albeit in the setting of Hermitian matrices) can give a different perspective on the NPPT bounded entanglement problem of quantum information theory (see for instance \cite[Section 4.4]{ela} and the references therein).

\medskip

This paper is organized as follows: in Section \ref{cb} we present the correspondence from norms to convex bodies in $\mathbb R^n$, the relation with majorization, and the family of \textit{orbit norms}. In Section \ref{fc}, by considering the dual norm $N'$, the correspondence among unit spheres of both norms is achieved by means of polar duality, using the trace functional. Therefore the faces of the unit sphere can be described by unit norm functionals, or equivalently, by elements of \textit{dual} unit norm in $\lu$ (of particular interest are those spheres which are polytopes such as the sphere of the trace norm, the sphere of the spectral norm, or the sphere of an orbit norm).  In Section \ref{fb} we obtain a precise characterization of these norming functionals, of the extreme points and the smooth points of the sphere. Then we prove the result (Theorem \ref{teoN}) which establishes necessary ans sufficient conditions for the inequality $\varphi([X,[X,V]])\le 0$ to be an equality (here $\varphi$ is a norming functional for $V$). As an application of the tools developed, in  Section \ref{action} we study in further detail the adjoint action $V\mapsto V+[X,V]$ of $\lu$ on itself: we show that this action always pushes vectors away from the unit sphere, and it only preserves the norm under strict conditions. These conditions imply in particular that for a strictly convex norm, we have $\|V+[X,V]\|\ge \|V\|$, with equality if and only if $[X,V]=0$. 

\smallskip

With these tools, the geometry of the unitary group $\U$ equipped with the left-invariant metric induced by a Finsler norm $N$ will be studied in a follow-up paper -in preparation-, following ideas from \cite{alr,cocoeste,alv,atkin1,dmlr0,lar19,larmi}.

\section{Ad-invariant Finsler norms in $\lu$}

Throughout this work, $\U$ will denote the unitary matrix group $\U=\mathsf{\mathbf U}(n,\mathbb C)=\{U\in M_n(\mathbb C): U^{-1}=U^*\}$, and $\lu=\mathfrak u_n(\mathbb C)$ is the Lie algebra of the Lie group $\U$, the real linear space of skew-Hermitian matrices, with real dimension  $dim_{\mathbb R}(\lu)=n^2$. We will also denote  $\slu=\mathfrak su_n(\mathbb C)$ the traceless skew-Hermitian matrices. We will be interested in norms invariant for the coadjoint action of the group, and in fact we will also deal with a broader class, the Finsler norms:

\begin{defi}[$\Ad$-invariant Finsler norm in $\lu$] It is a function $\|\cdot\|:\lu\to\mathbb R_{\ge 0}$ such that for any $X,Y\in\lu$
\begin{enumerate}
\item $\|X\|=0$ if and only if $X=0$ (non-degenerate)
\item $\|X+Y\|\le \|X\|+\|Y\|$ (subadditive)
\item $\|\lambda X\|=\lambda \|X\|$ for any $\lambda\in \mathbb R_{\ge 0}$ (positively homogenous)
\item $\|UXU^*\|=\|X\|$ for any $U\in \U$ ($\Ad$-invariant).
\end{enumerate}
Note that we do not require smoothness neither full  homogeneity. Even though a Finsler norm is not necessarily a  norm, it induces a topology in $\lu$, which is given by the notion of convergence:
$x_{\alpha}\longrightarrow x  \Leftrightarrow \|x_{\alpha}-x\|\longrightarrow 0$. 
\end{defi}

\begin{defi} In standard matrix theory books, a norm $||\cdot ||$ defined in $M_n(\C)$ is called  weakly invariant if $||UXU^*||=||X||$ for every $X\in M_n(\C)$ and $U\in \U$, and strongly invariant if $||UXV||=||X||$ for every $X\in M_n(\C)$ and $U,V\in \U$. 
\end{defi}

\begin{rem}[Minkowski gauges]\label{minkowsky} Let $\|\cdot\|$ be an $\Ad$-invariant Finsler norm in $\lu$ then, for $r>0$, the set $B_r=\{V\in\lu: \|V\|<r\}$ is open, absorbing, convex and Ad-invariant. Reciprocally, if $B$ is an absorbing, convex and Ad-invariant set around $0\in \lu$ then its Minkowski gauge function $\rho_B$  defines an Ad-invariant Finsler semi-norm ($\rho_B(X)$ might be zero for $X\neq 0$) in $\lu$. Moreover if $B$ does not contain rays from zero, then the Minkowski gauge defines an Ad-invariant Finsler norm ($\rho(X)=0$ implies $X=0$). 
\end{rem}

\begin{rem}\label{normal} Let $ip\in \lu$ with $p$ a one-dimensional orthogonal projection. Then for every one-dimensional projection $q$ we have $q=Upu^*$ for some $U\in \U$, thus $\|iq\|=\|UipU^*\|=\|ip\|$; we choose the normalization $\|ip\|=1$. All Ad-invariant Finsler norms in $\lu$ are equivalent, and furthermore, they are equivalent to any other norm, in particular they all induce the same topology.  Every weakly invariant norm in $M_n(\mathbb C)$ induces by restriction, an $\Ad$-invariant norm in $\lu$. Reciprocally, for any $\Ad$-invariant norm in $\lu$ we can construct one (in fact, many) weakly invariant norm in $M_n(\mathbb C)$ such that its restriction to $\lu$ is the original norm, see Remark \ref{taylor}.
\end{rem}

\begin{ejem} A short list with some relevant $\Ad$-invariant norms follows: 
\begin{enumerate} 
\item $\|X\|_1=\tr |X|$ (trace norm).
\item $\|X\|_{\infty}=\max \{|\lambda|:\lambda\in\sigma (X)\}$ (spectral norm).
\item For $1<p<\infty$, $\|X\|_p=\left(\tr|X|^p\right)^{\frac{1}{p}}$ ($p$-norms).
\item $\|X\|_{(k)}=\sum\limits_{1\leq i\leq k} |\lambda_i|$, where $\lambda_i$ are the eigenvalues of $X$ ordered decreasingly in absolute value (Ky-Fan norms).
\item $\|X\|_C=\max \{|\tr(CU^*XU)|:U\in \U\}$  for $C\in M_n(\C)$ a non scalar matrix with $\tr(C)\neq 0$ (C-radius norms).
\end{enumerate}
\end{ejem}

We have $\|\cdot\|_{KF(1)}=\|\cdot\|_{\infty}$, $\|\cdot\|_{KF(n)}=\|\cdot\|_{1}$, and $\|X\|_2=\|X\|_F$ (the Frobenius norm). All of these examples are in fact norms (fully homogeneous), and all but the last one are in fact strongly unitarily invariant; a good reference on the subject is Bhatia's book \cite[Chapter IV]{bhatia}. In fact, the proof that the last one is a norm can be found in \cite[Proposition IV.4.4]{bhatia}. Moreover, it is well-known that if $C$ is any one-dimensional projection then the $C$-norm is the numerical radius, hence  it is equal to the spectral norm in $\lu$. 

\begin{rem} For $A,B\in M_n(\R)$ the bilinear application given by 
$$(A |B):=\tr(AB^*)$$
defines a real inner product, known as the Hilbert-Schmidt inner product. The Frobenius norm $||\cdot ||_F$ is the norm induced by this inner product. In fact, the Frobenius norm and it's positive multiples, are the only $\Ad$-invariant norms defined in $\lu$ that are induced by an inner product.
\end{rem}

\begin{rem}[Trace duality, norming functionals]\label{trduality}We will use the duality induced by the trace of $M_n(\mathbb C)$ to identify the dual space of $(\lu,\|\cdot\|)$ with $\lu$ as follows:  for $\varphi\in \lu'$ (the dual space of $\lu$) consider 
$$
\|\varphi\|=\max\{\varphi(X): \|X\|\le 1\}
$$
(without the modulus); it is easy to check that this defines a Finsler norm in the dual space. For $V\in\lu$ define
\begin{equation}\label{normi}
\varphi_V(X)=(V|X)=\tr(VX^*)=-\tr(VX),
\end{equation}
then $\varphi_V\in \lu'$. Since $(\cdot |\cdot)$ is non-degenerate, each non-zero $\varphi\in\lu'$ comes exactly from one non-zero $V\in \lu$ in this fashion.  We say that $\varphi\in \lu'$ is a \textit{unit norming functional} of $V\in\lu$ if $\varphi(V)=\|V\|$ and $\|\varphi\|=1$. By the Hahn-Banach theorem each non-zero vector $V\in\lu$ admits at least one ot these; for brevity we say that $\varphi$ is \textit{norming} for $V$. Note that 
\begin{align*}
(V|\,[X,Y]) & =-\tr(V[X,Y])=-\tr(VXY-VYX)=-\tr(XYV-XVY)\\
& = -\tr(X(YV-VY))=-\tr(YVX-YXV)=-\tr(Y(VX-XV))
\end{align*}
by the cyclicity of the trace. Therefore we obtain the cyclic identities
\begin{equation}\label{cyclic}
(V|\,[X,Y])=(Y|\,[V,X])=(X|\,[Y,V])
\end{equation}
for any $X,Y,V\in \lu$. These can be rewritten as $(V\,|\ad X(Y))=-(\ad X(V)|Y) =(\ad Y(V)\,|\,X)$. In particular $\ad X:\lu\to \lu$ is skew-adjoint for the Hilbert-Schmidt inner product, for any $X\in\lu$ (which implies that $-\ad^2(X)\ge 0$ hence $\|V+[X,[X,V]]\|_F\ge \|V\|_F$, see Theorem \ref{vmasxv} below for a generalization). 
\end{rem}

\subsection{Convex bodies and norms}\label{cb}

Next we state a version of the theorem of Horn and Uhlmann regarding the convex hull of the orbit of a self-adjoint matrix \cite{horn};  the equivalences for fully homogeneous norms can be found in Bhatia's book \cite[Theorem IV.4.7]{bhatia}.  First we need some notations:

\begin{rem}[Notations]Any matrix $X\in \lu$ can be written as $X=UDU^*$, where $U\in \U$ and $D$ is a diagonal matrix with purely imaginary entries. Hence the eigenvalues and singular values of $X$ are the same in absolute value, i.e. $|\lambda_i(X)|=|x_i|=\lambda_i(|X|)$. We denote $\overrightarrow{x}=(x_1,x_2,\dots,x_n)$ the string of real numbers such that $i x_j$ are the eigenvalues of $X$. For $B\subset\lu$ we denote with  $\OO(B)$ the coadjoint orbit of $B$ in $\lu$ (for the action of the unitary group $\U$), i.e.
$$
\OO(B)=\{UXU^*: X\in B, \, U\in \U\}.
$$
If $C\in\lu$, we denote $\OO_C=\OO(\{C\})$ for short. We denote with $\co(\Sigma)$  the convex hull of the set $\Sigma$. 
\end{rem} 

With these notations, we can now state the equivalences; see Section \ref{fc} for definitions and discussion of faces of the sphere:

\begin{prop}[Majorization and Finsler norms]\label{majo}
Let $Z,W\in\lu$. The following are equivalent:
\begin{enumerate}
\item $Z\in \co(\OO_W)$, more precisely there exist (at most) $n+1$ matrices  $U_i\in \U$ and $n+1$ real numbers $\lambda_i \ge 0$ with $\sum_i \lambda_i= 1$ such that $Z=\sum_{i=1}^{n+1} \lambda_i\, U_iWU_i^*$.  
\item $\overrightarrow{z}\prec \overrightarrow{w}$ (strong majorization).
\item $\|Z\|\le \|W\|$ for all $\Ad$-invariant Finsler norms in $\lu$.
\item $\max_{U\in \U} (UCU^*|Z)\le \max_{U\in \U} (UCU^*|W)$ for all $C\in \lu$.  
\end{enumerate}
If moreover equality holds for some Finsler norm, then $Z$ and all the $U_iWU_i^*$ lie in the same face of the sphere for that norm (and in fact lie in the intersection of all the faces such that $Z$ lies in). If that norm is strictly convex then $Z=UWU^*$ for some $U\in \U$.
\end{prop}
\begin{proof} The proof is similar to that of \cite[Proposition 5.8]{larmi}, with the obvious modifications, and therefore omitted.
\end{proof}

\begin{rem}
It suffices to check condition $(3)$ above only for \textit{strictly convex norms} to obtain the equivalences. This is because any $\Ad$-invariant Finsler norm  $\|\cdot\|$ can be approximated explicitly with a strictly convex ($\Ad$-invariant, Finsler) norm by means of $\|x\|_{\varepsilon}=\|x\|+\varepsilon \|x\|_F$.  On the other hand, it suffices to check $(4)$ for \textit{regular} $C\in \lu$ (that is, all the eigenvalues of $X$ are different), because that is a dense set; if moreover we know a priori that $\tr(Z)=\tr(W)$ (which is a necessary condition), then it suffices to check $(4)$ for regular $C\in \slu$.
\end{rem}

\begin{rem}[Orbit norms in $\slu$]\label{orbitnorms}
If $X\in \slu$, then $\sum_j x_j=-i\tr(X)=0$. This easily implies that all the partial sums $\sum_{k=1}^m x_k$, with the $x_k$ rearranged in decreasing order, must be non-negative. Thus the vector $\overrightarrow{x}$ strongly majorizes the zero vector in $\mathbb R^n$, i.e. $\overrightarrow{0}\prec \overrightarrow{x}$, and by the previous proposition,  $0\in \co(\OO_X)$ for any $X\in \slu$. Then for fixed $C\in \lu$
$$
\max\{(UCU^*|X): U\in \U\}=\max\{-\tr(CU^*XU): U\in \U\}\ge 0
$$
for any $X\in \slu$. For a fixed $C\in \lu$ such that $C\ne \lambda 1$, we define the \textit{orbit norm} in $\slu$  as
$$
\|X\|_{\OO_C}=\max\{(UCU^*|X): U\in \U\}.
$$
Note that if $C=\lambda 1$ then $\|X\|_{\OO_C}=\tr(X)=0$ for all $X\in\slu$, thus the condition is necessary to obtain a Finsler norm. Next we show that it is sufficient. Since $X\in \slu$, we can replace $C$ with $C-\tr(C/n)1$ and obtain the same quantity, thus we will  assumme that $C\in \slu$. 
\end{rem}

Note that if $C$ is a fixed point of the action of the unitary group on $\lu$, then $-C$ is also a fixed point for the action. Morevoer, there  are no fixed points or there are exactly two of the, $C$ and $-C$.

\begin{lem}\label{twoorbits} Let $0\ne C\in \slu$. Then the orbit norm is an $\Ad$-invariant Finsler norm in $\slu$. Moreover:
\begin{enumerate}
\item The norm of $X$ is attained at $UCU^*$ which commutes with $X$, and with the eigenvalues of both arranged in decreasing order.
\item $C$ belongs to the unit sphere of the norm if and only if  $\|C\|_F=1$.
\item In that case, $\|X\|_{\OO_C}=1$ when $X\in \OO_C$, and for $X\in \co(\OO_C)$ we have  $\|X\|_{\OO_C}\le 1$ with equality if and only if $X\in \OO_C$.
\item The norm is fully homogeneous if and only if $\sigma(C)=-\sigma(C)$.
\end{enumerate}
\end{lem}
\begin{proof}
By the previous discussion, the orbit norm is non-negative, and it is clear that it is positively homogeneous. It is also clear that the triangle inequality holds and that it is $\Ad$-invariant. So we have to check that it is non-degenerate. Let $\overrightarrow{c}$ denote the string of eigenvalues of $-iC$. Since $C$ is not a multiple of the identity map, and $\tr(C)=0$, it is well-known that there exist $n-1$ permutations $\sigma_k\in S_n$ such that $\{\sigma_k(\overrightarrow{c})\}_{k=1,\dots,n-1}$ spans $\{v\in\mathbb R^n:\sum_{i=1}^n v_i=0\}$ in $\mathbb R^n$ (see for instance \cite{hema}). Hence the orbit of $C$ spans $\slu$. We know that $0\in \co(\OO_C)$, we claim that it is an interior point. If that is not the case, we use the argument in \cite[Lemma 6]{bgh}: there exist by the Hahn-Banach separation theorem a unit norm functional $\varphi$ such that $\varphi(0)=0$ and $\varphi(\co(\OO_C))\ge 0$. Then $\varphi(UCU^*)\ge 0$ for all $U\in \U$. Let $dU$ be the normalized Haar measure in $\U$, let $W=\int_{\U}UCU^*dU$. Then $W\in \slu$ and $\Ad_U W=W$ for all $U\in\U$, and this is only possible if $W=\lambda 1$, thus $W=0$. Now 
$$
0=\varphi(W)=\int_{\U}\varphi(UCU^*)dU\ge 0
$$
and this is only possible if $\varphi(UCU^*)=0$ for all $U$. But this implies that $\co(\OO_C)\subset\ker\varphi$, a contradiction. Now let $X\ne 0$ in $\slu$, shrinking $X$ we obtain an element of the interior of the convex capsule of the orbit of $C$, hence there exists $\lambda>0$ such that $\lambda X=\sum_i \lambda_i U_i CU_i^*$. Assume that $\|X\|_{\OO_C}=0$, then $(UCU^*|X)\le 0$ for all $U\in \U$. But on the other hand
$$
0<\lambda \|X\|_F^2 =\sum_i \lambda_i (U_iCU_i^*|X)\le 0,
$$
a contradiction. Hence it must be $\|X\|_{\OO_C}\ne 0$, and this finishes the  proof of the fact that the orbit norm is a true Finsler norm in $\slu$. Let $Z\in\slu$, let $U$ such that $(UCU^*|X)=\|X\|_{\OO_C}$, then $f(t)=(e^{tZ}UCU^*e^{-tZ}|X)$ has a maximum at $t=0$, hence
$$
0=f'(0)=([Z,UCU^*]|X)=(Z|[UCU*,X])
$$ 
by (\ref{cyclic}), hence $[UCU^*,X]=0$, and we can diagonalize both simultaneously hence $(UCU^*|X)=\sum_i c_{\sigma(i)}x_i$ for some permutation $\sigma\in S_n$. But any such sum is dominated by the sum with both strings of coefficients in decreasing order (see e.g. \cite[Corollary II.4.4]{bhatia}). This proves claim 1. For the normalization condition 2., note that
$$
-\tr(UCU^*C)=\tr(UCU^*C^*)\le \|UCU^*\|_F \|C\|_F=\|C\|_F^2=-\tr(C^2)
$$
by the Cauchy-Schwarz inequality for the trace inner product. Hence $\|C\|_{\OO_C}\le -\tr(C^2)$ and on the other hand by picking $U=1$ we also have $\|C\|_{\OO_C}\ge -\tr(C^2)$. This shows that $\|C\|_{\OO_C}=-\tr(C^2)=\|C\|_F^2$. So let us fix $C\in\slu$ such that $\|C\|_F=1$. If $X\in \OO_C$, then from the definition of the norm we have  $\|X\|_{\OO_C}=\|C\|_{\OO_C}=1$, and if $X\in \co(\OO_C)$ it is then apparent that $\|X\|_{\OO_C}\le 1$. Assume that $X$ is in the convex hull of the orbit and $X$ has unit norm, i.e.  $X=\sum_i \lambda_i U_iCU_i^*$ and there exists $U_0$ such that $(U_0C U_0^*|X)=\|X\|_{\OO}=1$. Then $
-CU_0^*XU_0=\sum_i \lambda_i (-C)\tilde{U_i}C\tilde{U_i}^*$, and taking trace and using the Cauchy-Schwarz inequality for the trace we have 
$$
1=(C|U_0^*XU_0)= \sum_i \lambda_i |(C|\tilde{U_i}C\tilde{U_i}^*)\le \sum_i \lambda_i\,\|C\|_F^2=1.
$$
Therefore it must be $(C|\tilde{U_i}C\tilde{U_i}^*)=\|C\|_F\|\tilde{U_i}C\tilde{U_i}^*\|_F$ for each $i$, and the equality for the Cauchy-Schwarz inequality for the trace is attained if and only if there exists $t_i\ge 0$ such that $C=t_i \tilde{U_i}C\tilde{U_i}^*$. By the equality that must hold for each $i$, and since $\|\tilde{U_i}C\tilde{U_i}^*\|_F=\|C\|_F=1$, it must be $t_i=1$ for each $i$. Thus $C$ commutes with $\tilde{U_i}$ for each $i$, and $U_0^*XU_0=\sum_i \lambda_i \tilde{U_i}C\tilde{U_i}^*=C$, which shows that $X=U_0CU_0^*\in \OO_C$, finishing the proof of 3. Now let us assume that the spectrum of $C$ is balanced, then if $X\in\slu$ pick $U$ such $(UCU^*|X)=\|X\|_{\OO_C}$ and $U_0$ such  $U_0CU_0^*=-C$. Then
$$
\|-X\|_{\OO_C}\ge (UU_0CU_0^*U^*|-X)=(UCU^*|X)=\|X\|_{\OO_C}.
$$
Argumenting with $-X$ we obtain the opposite inequality. Finally, if the spectrum of $C$ is not balanced, then $-C\notin \OO_C$ (otherwise $-C=UCU^*$ is a contradiction); let us asumme that $\|-X\|_{\OO_C}=\|X\|_{\OO_C}$ for any $X\in \slu$, then if $p$ is any one-dimensional projection we have $c_1=\|ip\|_{\OO_C}=\|-ip\|_{\OO_C}=-c_n$. Now let $p_1,p_2$ be orthogonal one-dimensional projections, then by assertion 1. of this lemma we have
$$
c_1+c_2=\|i(p_1+p_2)\|_{\OO_C}=\|-i(p_1+p_2)\|_{\OO_C}=-c_{n-1}-c_n,
$$
thus $c_2=-c_{n-1}$. We proceed in this fashion until we end up with $c_{k+1}=0$ (for $n=2k+1$ odd), or until we end up with $c_k=-c_{k+1}$ (for $n=2k$ even). In either case, $\sigma(C)=-\sigma(C)$.
\end{proof}

Even though the natural setting for the orbit norms is $\slu$, we can extend them to $\lu$ in several $\Ad$-invariant ways, for instace:

\begin{coro}[Orbit norms in $\lu$] Let $0\ne C\in \slu$, then 
$$
\|X\|_{\OO_C}=\max\{(UCU^*|X): U\in \U\}+|\tr X|
$$
is an $\Ad$-invariant Finsler norm in $\lu$, and it is fully homogeneous if and only if the spectrum of $C$ is balanced. 
\end{coro}
\begin{proof}
It is plain that it is non-negative and positively homogeneous; the triangle inequality is also easy to establish, and so is the invariance. Now if the result is zero, since both terms are non-negative,  it must be $X-\tr(X/n)1=0$ and $\tr X=0$ by the previous lemma, hence $X=0$. The last assertion has a simple proof using the previous lemma, and is therefore omitted.
\end{proof}

\begin{rem}Let $\overrightarrow{z},\overrightarrow{w}$ denote the strings of eigenvalues of $-iZ,-iW$ with $Z,W\in\lu$. Combining Proposition \ref{majo} and Lemma \ref{orbitnorms}, we have $\|Z\|\le \|W\|$ for all Finsler $\Ad$-invariant norms if and only if $\tr(Z)=\tr(W)$ and
$$
\sum_i c_i^{\downarrow}\;z_i^{\downarrow}\le \sum_i c_i^{\downarrow}\; w_i^{\downarrow}
$$
for any string  $\overrightarrow{c}$ such that $\sum_i c_i=0$ and all the $c_i$ are different (the down arrow indicates that the eigenvalues are ordered decreasingly).
\end{rem}

\medskip

We now discuss the presentation of a Finsler norm as the Minkowski norm of the coadjoint orbit of a symmetric set:


\begin{rem}\label{masaobs} Let $\MM\subset \lu$ be a maximal abelian associative subalgebra (m.a.s.a.), then there exist an orthonormal basis $B$ such that $\MM=\{A\in\lu: A \hspace{0.2cm}\text{is diagonal in the base B}\}$.   All such algebras are conjugated by a unitary matrix. For any given $\mathfrak M$, consider the set of $n$  one-dimensional projections $(p_k)_{k=1}^n$ generating $\mathfrak M$. If we identify $i p_k$ with $e_k$ (the canonical vectors in $\mathbb R^n$), then we  identify as before any diagonal matrix $A=i\sum a_k p_k\in\lu$ with a vector in $\mathbb R^n$, $\vec{a}=(a_1,\cdots,a_n)$. With the Frobenius norm in $\lu$, this identification is isometric if we put in $\mathbb R^n$ the canonical inner product.
\end{rem}

\begin{defi} A set $B$ around $0\in\mathbb R^n$ is \textit{symmetric} if for every permutation $\sigma\in S_n$ and $(x_1,x_2, \cdots ,x_n)\in B$, we have $(x_{\sigma(1)},x_{\sigma(2)}, \cdots ,x_{\sigma(n)})\in B$.  If $F:\mathbb R^n\to\mathbb R$, we say that $F$ is \textit{symmetric} if $F\circ\sigma=F$ where $\sigma$ should be interpreted as a permutation of the variables. Note that $F$ is symmetric if and only if the level sets of $F$ are symmetric. 
\end{defi}

\begin{prop}\label{bolas}
Let $\|\cdot\|$ be an Ad-invariant Finsler norm in $\lu$, and $B_1$ it's open unit ball. Then
\begin{enumerate}
\item For every m.a.s.a. $\MM \subseteq \lu$, let $B_{\MM}=B_1\cap \MM$. Then  $B_{\MM}$ is a symmetric open convex set around $0\in \R^n$ (with the identification of the previous remark), not depending on the chosen algebra. 
\item Let $B\subset\R^n$ be an open convex symmetric set around $0$, and 
$$B_1=\co(\OO(B))=\{\sum_i \lambda_iU_iXU_i^* :X\in B, U_i\in\U, \lambda_i\geq 0, \sum_i\lambda_i=1\}.$$
Then $B_1$ is an open, Ad-invariant convex set around $0\in \lu$, and for every m.a.s.a $\MM$ we have  $B_1\cap \MM=B$ (using the identification of $\MM$ with $\R^n$). 
\end{enumerate}
\end{prop}
\begin{proof}
Since $B_1$ is open and convex,  then $B_{\MM}=B_1\cap \MM\subset \MM \simeq \R^n$ is an open convex set. Let $\MM_1$ be another m.a.s.a., then there exists a unitary matrix $U$ such that $U\MM U^*=\MM_1$. Since
$$
B_{\MM_1}=B_1\cap \MM_1=UB_1U^*\cap U\MM U^*=U(B_1\cap \MM) U^*= UB_\MM U^*,
$$
this shows that $B_\MM$ and $B_{\MM_1}$ (when looked as subsets of $\mathbb R^n$) are the same set. The set is clearly symmetric, since we can act with $\U$ by permutations. This proves the first assertion. Now let $B\subset\R^n$ be an open convex symmetric set around $0$, and present it as a subset of a m.a.s.a. $\MM$. From the definition of $B_1$ we see that it is open, convex, $\Ad$-invariant and contains the $0$ matrix. We only need to prove that $B_1\cap \MM_1=B$, since then for any given m.a.s.a. $\MM_1=U\MM U^*$ we will have 
$$
B_1\cap \MM_1=U(B_1\cap \MM)U^*=UBU^*,
$$
hence $B$ and $UBU^*$ are in fact the same set when we make use of our identification with $\mathbb R^n$.  So let $X\in B$, then $UXU^*\in B_1$ for all $U\in \U$, and  if we choose $U=1$ it follows that $X\in B_1$. On the other hand we choose $\MM$ such that $B\subset \MM$, and then $B\subset B_1\cap \MM_1$.  Since $B$ is a convex set, in order to prove that $B_1\cap\MM\subset B$ we need to show that if $Z=\sum\limits_i\lambda_iU_iXU_i^*\in B_1\cap \MM$.
with $X\in B$, $U_i\in \U$, $\lambda_i\ge 0$ and $\sum\lambda_i=1$ then $Z\in B$. We have $Z=i\sum z_k p_k$. Let $X^i=diag(U_iXU_i^*)=\sum_k p_k U_iXU_i^*p_k$ be the diagonal of $U_iXU_i^*$ with respect ot the given m.a.s.a. By Schur-Horn's theorem, each $X^i$ is majorized by $X$ (see \cite[Theorem 1.3]{ando}), or equivalently, there exists a doubly stochastic matrix such that $\vec{X^i}=A^i\vec{X}$. Now since $Z$ is diagonal,
$$
Z=i\sum z_k p_k=\sum_k p_kZp_k=\sum_i \lambda_i \sum_k p_kU_iXU_i^*p_k=\sum_i\lambda_i X^i.
$$
Let $A=\sum_i \lambda_i A^i$, then it is easy to check that $A$ is doubly stochastic, and since $\vec{Z}=\sum_i \lambda_i \vec{X^i}=\sum_i \lambda_i A^i\vec{X}=A\vec{X}$, we conclude that $\vec{Z}$ is majorized by $\vec{X}$. But then $\vec{Z}$ is a convex combination of coordinate permutations of $\vec{X}$, which shows that $Z\in B$.
\end{proof}

\begin{rem} If we start with $B\subset \R^n$ an open, symmetric set around $0$, then $B_1$ is an open, convex and Ad-invariant set of $\lu$. Thus, by Remark \ref{minkowsky}, it defines a Finsler norm in $\lu$. This fact and the last proposition imply that every $\Ad$-invariant Finsler norm in $\lu$ is obtained in this fashion.  Clearly, the same assertions hold for $\slu$, replacing $\mathbb R^n$ with the hyperplane $\sum_{i=1}^n x_i=0$.
\end{rem}

\begin{ejem}\label{ejemnor} We present here a couple of elementary examples that we will use later to ilustrate smootheness and duality:
\noindent 
\begin{enumerate}[wide=0pt]
\item\label{ejem1} \textit{The twisted Ellipse}: consider the (twisted) planar ellipse given by 
$$
\frac{1}{a^2}(x+y)^2+\frac{1}{b^2}(x-y)^2\le 1.
$$
This ellipse is a symmetric convex body with smooth (and strictly convex) boundary. By Proposition \ref{bolas}, the convex hull of the unitary orbit of this set is the unit ball of an $\Ad$-invariant Finsler norm in $\lu$, for $n=2$. The Minkowski functional of this convex body is given by
$$
F(x,y)=\sqrt{\frac{1}{a^2}(x+y)^2+\frac{1}{b^2}(x-y)^2},
$$
which is a Finsler norm in $\mathbb R^2$ (and therefore in $\lu$ for $n=2$): if $ix_k$ are the eigenvalues of $X$,  then $\|X\|=F(x_1,x_2)$.
\item \textit{The Toast}: consider the line segments $y=x+1$, joining the points $(-1,0)$ and $(0,1)$, then the line segment $y=x-1$ joining the points $(0,-1)$ and $(1,0)$, and the line segment $y=-x-1$ joining the points $(-1,0)$ and $(0,-1)$. Close this box with the circumference
\begin{equation}\label{toast}
(x-\frac{1}{2})^2+(y-\frac{1}{2})^2=\frac{1}{2},
\end{equation}
joining the points $(1,0)$ and $(0,1)$. The area enclosed by this curves is a symmetric convex body in $\mathbb{R}^2$, a toast. The formula for the norm (the Minkowski functional of this convex body) is apparent if $(x,y)$ belongs to the second, the third and the fourth quadrant: it is given by the $1$-norm: $\|(x,y)\|=|x|+|y|$. On the other hand, if $x,y\ge 0$, we can rewrite equation (\ref{toast}) as $x^2+y^2\le x+y$. Therefore in the first quadrant the Finsler norm is given by
$$
\|(x,y)\|=\frac{x^2+y^2}{x+y}.
$$
The point $(1,0)$ where the circumference joins (smoothly) the segment is an extreme point of the sphere but there is no linear functional supporting only that point.
\end{enumerate}
\end{ejem}

\subsection{Faces of the sphere, extreme points and dual norm}\label{fc}

Let $\overline{B_1}$ be the closed unit ball of the norm. Being a compact convex set it is by the Krein-Milman theorem the convex hull of its extreme points.

\begin{defi}[Extreme points, face and the cone generated by a face]\label{caras}
Recall $0\ne X\in \lu$ is \textit{extreme} if $X/\|X\|$ is an extreme point of $\overline{B_1}$. Note that the norm is strictly convex if and only if all the non-zero vectors are extreme. A \textit{face} $F$ of the ball $B_r$ of the normed space $(\lu,\|\cdot\|)$ is the intersection of the closed ball $\overline{B_r}$ with the hyperplane determined by a unit norm functional $\varphi\in \lu', \|\varphi\|=1$, i.e.
$$
F_{\varphi}(r)=\overline{B_r}\cap\{X\in\lu:\varphi(X)=r\}.
$$
We will usually omit the number $r$ and $F_{\varphi}$ will refer to the face containing a certain vector $V$, thus $r=\|V\|$. We say that the face is \textit{maximal} if $\varphi$ is an extreme element of the dual space. Every face is contained in a maximal face: since the closed ball of the dual norm is compact and convex, there exists by the Krein-Milman theorem extreme functionals $\{\varphi_i\}_{i=1,\dots,n}$ of unit norm such that $\varphi$ is a convex combination of the $\varphi_i$
$$
\varphi=\sum_i \lambda_i \varphi_i,  \quad \lambda_i\ge 0,\quad \sum_i \lambda_i=1.
$$
It is then easy to check that if $\varphi(X)=\|X\|$ then $\varphi_i(X)=\|X\|$ for all $i$. Therefore if $X\in F_{\varphi}$, then $X\in F_{\varphi_i}$ for all $i$, and in fact $F_{\varphi}$ is the intersection of all the maximal faces that contain it. The \textit{cone generated} by a face $F_{\varphi}$ is $C_{\varphi}=\mathbb R_+ F_{\varphi}$. This cone consists exactly of those $X\in \lu$ such that $\varphi(X)=\|X\|$ for this given unit norm $\varphi$.
\end{defi}

Note that for each $\lambda\in\mathbb R_{>0}$, if $\varphi(V)=\|V\|$ and $\|\varphi\|=1$ we have $
\varphi(\lambda V)=\lambda \|V\|=\|\lambda V\|$ thus $\varphi$ norms the whole ray $\lambda V$, $\lambda >0$. We say that $\varphi$ is a \textit{norming functional for $X$} if $\|\varphi\|=1$ and $\varphi(X)=\|X\|$.

\begin{defi}[Dual norm]\label{dualnorm} The duality of Lemma \ref{trduality} also allows us to define the \textit{dual norm} of $V\in \lu$ as
$$
\|V\|'=\max\{(V|X): \|X\|\le 1\},
$$
and it is easy to check that this norm is also $\Ad$-invariant, and if $V\in\lu$ and $\varphi$ is the functional defined by $V$ as in (\ref{normi}), then $\|\varphi\|=\|V\|'$.
\end{defi}

\begin{rem} If $\|V\|=1$, for each norming $\varphi$ of $V$ we have $N\in \lu$ of unit \textit{dual norm}, i.e. $\|N\|'=1$, such that $\varphi(V)=(N|V)=1$. But then if $\varphi'=(V|\cdot)$, it is easy to check that since $\|V\|=1$, this functional is of unit norm in the double dual space. On the other hand $\varphi'(N)=(V|N)=\|V\|=1=\|N\|'$, therefore $\varphi'$ norms $N$ in the dual space. Let $F_{\varphi'}\subset \lu$ stand for the face of the sphere for the dual norm, given by the functional $\varphi'$. Clearly $\varphi=(N|\cdot)$ is the unique functional supporting $V$ if and only if $F_{\varphi'}=\{N\}$, and also $F_{\varphi}=\{V\}$ if and only if $\varphi'=(V|\cdot)$ is the unique functional supporting $N$ in the dual space.
\end{rem}


\begin{rem}[Convexity and smoothness]\label{convesmooth} It is straightforward from the definitions that the double dual norm equals the original norm. On the other hand, the following properties will be useful:
\begin{enumerate}
\item The norm $\|\cdot\|$ is G\^{a}teaux differentiable at $x\ne 0$ if and only if it is Fr\'echet differentiable  (this follows from \v{S}mulian Lemma, see \cite[Lemma 8.4]{fabian} for instance). Therefore we simply say that the norm is \textit{smooth} when this happens for any $x\ne 0$; in that case the norm function is in fact $C^1$ away from $x=0$ \cite[Corollary 8.5]{fabian}.
\item The norm is smooth if and only if the dual norm is strictly convex if and only if there is a unique functional supporting each $x\ne 0$ \cite[Lemma 8.4 and Fact 8.12]{fabian}.
\item The norm is strictly convex if and only if it is uniformly convex (this follows from the previous properties and the fact that double dual norm equals the original norm, or it can be proven directly from the compacity of the unit sphere).
\item If the norm is smooth at $X\ne 0$, the  gradient of the norm $F$ at $X$ (with respect to the trace inner product) returns $N\in \lu$ such that $\varphi=(N|\cdot)$ is the unique norming functional for $X$. This is apparent from 
$$
t^{-1}(\|X+tV\|-\|X\|)\le \|V\| \qquad \forall \, t>0,
$$
which implies $(\nabla F(X)|V)\le \|V\|$ for all $V$, and $(\nabla F(X)|X)=\|X\|$. In particular when the norm is smooth, $X\mapsto \nabla F(X)$ is a bijection from $\partial B_1$ onto $\partial B_1'$.
\item In any case, since the norm is a convex function,  the subdifferentials of the norm at $X\ne 0$, completely characterize the norming functionals of $X$ (see \cite[Section 2]{watson} and Theorem \ref{normderivative} below).
\end{enumerate}
See also \cite{fan} for further general properties of a spehere in a Banach space.
\end{rem}



\begin{rem}[Polar duals] Let $E$ be a convex subset in an inner product space $(V, (\cdot |\cdot )\,)$. The \textit{polar dual} of $E$ is the set
$$
E^{\circ}=\{Y\in V:  (Y|X)\le 1\quad \forall\, X\in E\}.
$$
It is plain that $E\subseteq (E^{\circ})^{\circ}$, in fact we have that $(E^{\circ})^{\circ}$ equals to the closure of the convex hull of $E\cup\{0\}$. If $E$ is closed and $0\in E$, since we are assumming that $E$ is convex, we have that $E=(E^{\circ})^{\circ}$ as a consequence of the Hahn-Banach theorem. For further discussion see \cite{jensen} and the references therein.  

Let $E=B_1$ be the unit ball of a Finsler norm $\|\cdot\|$ in $\lu$, let $B=B_1\cap \mathfrak M\subset\mathbb R^n$ be the $n$-dimensional convex set that generates $B_1$ (Proposition \ref{bolas}). It is apparent that the norm $\|\cdot\|$ is strictly convex  if and only if the boundary set $\partial B$ is strictly convex  in $\mathbb R^n$ (for details, see for instance \cite[Theorem 3.1]{zietak}. Then by duality and using Proposition \ref{bolas}, we have that the norm in $\lu$ is smooth if and only if $\partial B$ is smooth in $\mathbb R^n$. 
\end{rem}

\begin{defi}[Self-duals] By the very definition of dual norm in $\lu$, if $B_1'$ is the unit ball of the dual norm, then we have $B_1'=B_1^{\circ}\subset \lu$ with respect to the trace inner product in $\lu$. Moreover $B_1'=\co( B^{\circ})$ by the same proposition, where we take the polar dual of $B\subset\mathbb R^n$ with respect to the standard inner product there. It is not hard to check that the only self-polar closed convex body around $0$ is the unit ball of the Euclidean norm. But since we are interested in the geometry of the ball, we can allow rotations and dilations/contractions. So we will say that the norm is \textit{self-dual} if the polar dual of its unit ball is homothetic to a rotation of the original unit ball.
\end{defi}

\begin{ejem} We now discuss some dual norms and the geometry of their unit spheres. By the previous remark, it suffices to study the polar dual of the $n$-dimensional convex set $B\subset \mathbb R^n$ obtained by cutting the unit ball $B_1$ of the Finsler norm with any m.a.s.a. Being convex sets, they behave well with respect to the interior and closure operations, so we abuse notations a bit and use $B,B_1$ to indicate the closures of the balls also.
\begin{enumerate}
\item It is well-known that the dual of the $p$-norms are the $q$ norms with $q=1-1/p$, including the case $q=\infty$ when $p=1$ and vice-versa. For $1<p<\infty$, the unique $\varphi=(N|\cdot)$ norming $X=i\sum_k x_kp_k$ is given by 
$$
N=\frac{1}{\|X\|_p^{p-1}}\sum_k sg(x_k)x_k^{p-1}
$$
\item The unique self-polar norm is the Frobenius norm $(p=2)$. The $1$-norm and the $\infty$-norm are self-dual only for $n=2$, and the $p$-norms for $p\ne 1,2$ are not self-dual. 
\item From  Remark \ref{convesmooth}.4, a straightforward computation shows that the unit sphere of the dual norm of the twisted ellipse in Example \ref{ejemnor}, is yet another twisted ellipse, now with radii $2/a,2/b$. These norms are smooth and uniformly convex.
\item The dual unit ball of the toast in Example \ref{ejemnor} is obtained in a similar fashion: dualizing the straight diagonal segments, we obtain the three points $(-1,1),(-1,-1),(1,-1)$, the first one joined with the second with a vertical segment, and the second joined to the third one with an horizontal segment. Then $(-1,1)$ must be joined with $(1,-1)$ with the polar dual of the circumference; by computing its gradient we can see that it is the curve parametrized by $\alpha(t)=2(\cos t+\sin t +\sqrt{2})^{-1}(\cos t,\sin t)$ in the interval $t\in [-\pi/4,3\pi/4]$ (which is not a circumference).
\item The dual norm of the Ky-Fan norm $k$ is given by $\|X\|_{(k)}'=\max\left\{\frac{1}{k}\|X\|_1,\|X\|_{\infty}\right\}$, as shown in  \cite[Theorem 1]{watson1}. The norming functionals for the Ky-Fan norms $\|\cdot\|_{(k)}$ were completely characterized by Watson in \cite{watson1,watson} using the subgradient of the norm (Remark \ref{convesmooth}.5).
\item We point out one distinguished norming functional for the Ky-Fan norm $k$: let $X=i\sum_j x_j p_j\in \lu$ and write $X=\Omega |X|=|X|\Omega$ the polar decomposition of $X$. We can assume that $\Omega$ is unitary by taking 
$$
\Omega=i\sum_j sg(x_j)p_j\qquad |X|=\sum_j |x_j|p_j
$$
where $sg(x_j)=1$ if $x_j\ge 0$ and $sg(x_j)=-1$ if $x_j<0$. Then let $U$ be a unitary matrix rearranging the singular values of $|X|$ in decreasing order i.e. $U|X|U^*=\sum_j |x_j|^{\downarrow} p_j$ and  let $P=\sum_{j=1}^k p_j$. Consider $\varphi=(N|\cdot)$, with $N=-i \mathrm{Re}(U^*PU\Omega^*)\in \lu$. Then it is easy to check that $\varphi(X)=\|X\|_{(k)}$. On the other hand $\|\mathrm{Re} Z\|_p\le \|Z\|_p$ for any $p$ and any matrix $Z$, thus
$$
\|N\|_1\le \|U^*PU\Omega^*\|_1=\|P\|_1=k \qquad \|N\|_{\infty}\le \|U^*PU\Omega^*\|_{\infty}=\|P\|_{\infty}=1.
$$
By the previous remark we have $\|N\|_{(k)}'\le 1$, thus $\varphi$ is norming for $X$.
\end{enumerate}
\end{ejem}

Since the orbit norms have many interesting properties (which give good intuition for further general results on any norm, see the next section) we present their properties as a separate example. The natural setting is that of $\slu$, so we restrict to it:

\begin{ejem}[Orbit norms and their duals] The dual of the orbit norm of $C$ in  $\slu$ is the norm of the convex hull of the orbit, i.e. $B_1^{\circ}=B_1'=\co(\OO_C)$ (this is apparent from  from the definitions). Equivalently $B_1=\co(\OO_C)^{\circ}$, and in particular the maximum can be taken over the orbit or over the convex hull of the orbit:
$$
\|X\|_{\OO_C}=\max\limits_{V\in \OO_C}(X|V)=\max\limits_{Z\in \co(\OO_C)}(X|Z)\quad \textrm{ for any }X\in\slu.
$$
Another properties follow: 
\begin{enumerate}
\item By Schur-Horn's theorem, the extreme points of $B_1'\cap \mathfrak M$ are the vectors obtained by permutations of $ \overrightarrow{c}$ (as before, $\MM$ is any m.a.s.a).
\item The dual norm is not necessarily an orbit norm (equivalently, the unit ball of $\|\cdot\|_{\OO_C}$ is not necessarily the convex hull of a unitary orbit). An easy example of this follows: let $C=\diag(1,0,-1)$, then the convex hull of the orbit of $C$ in the hyperplane $x+y+z=0$ of $\mathbb R^3$ is a regular hexagon, this convex hull is $B_1'\cap \MM$, the unit ball of the dual norm sliced with the m.a.s.a. It is easy to see that the normal vectors to each of the sides of the hexagon are the vectors $(1,1,-2), (2,-1,-1)$ and its permutations. Hence the unit ball $B_1$ of the orbit norm $\|\cdot\|_{\OO_C}$ is the regular hexagon generated by these six points, but one cannot get these six points as the orbit of any given point. 
\item If $C$ is regular and balanced as in the previous example, then the orbit is a regular $n!$-agon, hence its polar dual is also a regular $n!$-agon. Hence in this case the norm is self-dual.
\item If $C$ is regular but not balanced, one obtains non-regular $n!$-agons (and the norm is not  self-dual). 
\item If $C$ is not regular, one obtains $k$-agons for certain divisors of $n!$. For instance if $C=(1,1,-2)$, the orbit of $C$ generates a regular triangle, hence its polar dual (the unit ball of the orbit norm) is also a regular triangle; hence this norm is self-dual.
\item With the normalization $\|C\|_F=1$ we have $\co(\OO_C)\subset B_1$ (Lemma \ref{twoorbits}), hence $B_1'\subset B_1$ thus $\|X\|_{\OO_C}'\ge \|X\|_{\OO_C}$ for any $X\in\slu$.
\item With this normalization, one has $\OO_C=\partial B_1\cap \partial B_1^{\circ}$. To prove this, recall that the unit ball of the dual norm is $\co(\OO_C)$, so we also have $\|C\|_{\OO_C}'\le 1$, thus $\|C\|_{\OO_C}'=1$. This tells us that the orbit $\OO_C$ is in the intersection of the unit sphere and the unit sphere of the polar dual of the ball. Now if $\|X\|_{\OO_C}=\|X\|_{\OO_C}'=1$ in particular $\|X\|_{\OO_C}'\le 1$ hence $X\in\co(\OO_C)$. Since $\|X\|_{\OO_C}=1$, we have $X\in \OO_C$ by Lemma \ref{twoorbits}.3.
\item With the above normalization, for any $U\in\U$, we have $\varphi=(UCU^*|\cdot)$ a unit norm functional for the orbit norm $\|\cdot\|_{\OO_C}$, since $\|UCU^*\|'=\|C\|'=1$. For any $X\ne 0$ there exists $U_0$ such that $(U_0CU_0^*|X)=\|X\|_{\OO_C}$, hence $\varphi=(U_0CU_0^*|\,\cdot\,)$ is norming for $X$. 
\end{enumerate}
\end{ejem}


\subsection{Norming functionals, the fine detail}\label{fb}

We now discuss more relevant properties of the matrix representing a norming functional for a fixed Finsler norm. 

\begin{lem}\label{diss}
Let $\|\cdot\|$ be an $\Ad$-invariant Finsler norm in $\lu$, let $0\ne X\in \lu$, let $\varphi=(N|\cdot)$ be a norming functional for $X$ or for $-X$. For $Y\in\lu$ we have that
\begin{enumerate}
\item  $\varphi([X,Y])=0$ (equivalently $\varphi \circ \ad X\equiv 0$ on $\lu$), and $N$ commutes with $X$.
\item If $\varphi$ norms $X$, then $\varphi([Y,[Y,X]])\le 0$.
\end{enumerate}
\end{lem}
\begin{proof}
Consider the expansion in $\lu$ 
\begin{align}
\Ad_{e^{sY}}X& =e^{s \ad  Y}X=X+s \, \ad  Y(X)+\frac{s^2}{2}[Y,[Y,X]]+ o(s^3),\label{expande2}
\end{align}
which gives us first: $e^{s \ad  Y} X-X=s[Y,X]+o(s^2)$. Then for each  norming functional $\varphi$ of $X$
$$
\varphi(e^{s\ad Y}X)-\varphi(X)\le \|e^{s\ad Y}X\|-\|X\|= \|e^{sY}Xe^{-sY}\|-\|X\|= 0.
$$
Divide by $s>0$ and make $s\to 0^+$, to obtain $\varphi([Y,X])\le 0$. Replacing $Y$ with $-Y$ proves that $\varphi([Y,X])=0$. The proof when $-\varphi(X)=\|-X\|$ is similar and therefore omitted.  Then for all $Z\in \lu$ we have 
$$
0=\varphi([X,Z])=-\tr(N[X,Z])=(N|[X,Z])=(Z|[N,X])
$$
were we used the cyclic identities (\ref{cyclic}). Taking $Z=[N,X]$ we obtain $\|[N,X]\|_F^2=0$, therefore $[N,X]=0$. Now by (\ref{expande2})
$$
o(s^3)+\frac{s^2}{2}\varphi([Y,[Y,X]])+0+\varphi(x)=\varphi(e^{s \ad Y}X)\le \|e^{s \ad Y}X\|=\|X\|=\varphi(X),
$$
therefore $o(s^3)+\frac{s^2}{2}\varphi([Y,[Y,X]])\le 0$. Dividing by $s^2$ and letting $s\to 0$ proves the second assertion. 
\end{proof}

\begin{rem}
Since $g: s\mapsto \varphi(e^{s\ad Y}X)$ has a maximum at $s=0$, then one possibility is that all the derivatives vanish at $s=0$; this can only happen if $\varphi(e^{s\ad Y})\equiv g(0)=\|X\|$ since $g$ is a real analytic function. If that is not the case, then all the expression $\varphi((\ad Y)^k(X))=0$ up to some finite \textit{even} term which must be strictly negative, i.e. $\varphi((\ad Y)^{2n}(X))<0$, and all the previous powers vanish.
\end{rem}

\begin{defi}[Diagonal/codiagonal decomposition] Let $0\ne V\in \lu$, write $V=i\sum_{k\in F} v_k P_k$ with $v_1>v_2>\dots >v_F$ a diagonalization of $V$ where the eigenvalues are different and ordered decreasingly. Each $P_k$ is an orthonormal projection onto the corresponding eigenspace and $P_kP_j=0$ for $k\ne j$. For any $X\in \lu$ we can decompose $X=X_D+X_C$, where
$$
X_D=\sum_k P_kXP_k\quad \textrm{ and  } \; X_C=\sum_{k\ne j}P_kXP_j.
$$
the \textit{block-diagonal and block-codiagonal} parts of $X$. A word of caution: it might easily be that $[X_D,X_C]\ne 0$. Note however that
$$
X=X_D+X_C,\qquad [X,V]=[X_C,V]\quad \textrm{ and }\; [X_D,V]=0.
$$
Let $\varphi=(N|\cdot)=-\tr(N\cdot)$ be a norming functional for $V$. Since $N$ commutes with $V$ (Lemma \ref{diss}), we have $N=i\sum_{k\in F}N_k$ with $N_k$ self-adjoint and $N_kN_j=0$ for $k\ne j$ ($N$ is block-diagonal). Next we extend the characterization of norming functionals stated for the orbit norms (Lemma \ref{twoorbits}.1) to any $\Ad$-invariant Finsler norm:
\end{defi}

\begin{prop}[Characterization of norming functionals]\label{ordenados} Consider $0\ne V=i\sum_{k\in F} v_k P_k$, and let $N=i\sum_k N_k$ with $\varphi=(N|\cdot)$ a norming functional for $V$, let $F_{\varphi}$ the face given by $\varphi$. Then
\begin{enumerate}
\item If $j>k$, any eigenvalue of $N_k$ is greater or equal than all the eigenvalues of $N_j$.
\item If a pair of these are equal then permuting the induced one-dimensional projections gives another $\tilde{V}\in F_{\varphi}$.
\item If $F_{\varphi}=\{V\}$, then in fact the inequality stated in the first item is always strict.
\item For given $s\in F$, exchanging the one-dimensional projections of $N_s$ corresponding to different eigenvalues, gives another $\phi$ norming $V$.
\item If $\varphi$ is the unique functional norming $V$, then $N$ is diagonal, i.e. $N_k=\sum_{k\in F} n_k P_k$, with $n_{k+1}\ge n_k$. 
\end{enumerate} 
In particular if the norm is strictly convex and smooth, then it must be $N=\sum_{k\in F} n_k P_k$ with the $n_k$ strictly decreasing.
\end{prop}
\begin{proof}
We have 
$\|V\|=\varphi(V)=(N|V)=-\tr(NV)=-\tr(i\sum_s N_s\, i\,\sum_l v_l P_l)=\sum_s v_s\tr(N_s)$. Now if $U\in \U$, then 
\begin{equation}\label{menor}
-\tr(U^*NUV)=-\tr(NUVU^*)=\varphi(UVU^*)\le \|UVU^*\|=\|V\|=-\tr(NV).
\end{equation}
Note also that if $\tilde{V}=UVU^*$ then $\varphi(\tilde{V})\le \|\tilde{V}\|=\|V\|$. Now choose an orthonormal basis $\{e_1^s,\dots,e_{d(s)}^s\}$ of $\ran(P_s)$ diagonalizing each $N_s$ (here $d(s)=dim(P_s)$) i.e. we have 
$$
N_s e_l^s= \lambda_l(N_s)e_l^s\qquad l=1,\dots,d(s)
$$
for each $s\in F$; we can further assume that the $\lambda_l(N_s)$ are in non-increasing order. Since $\sum_s P_s=1$, if we put these basis all together in the given order, we obtain a basis that diagonalizes $N$ and $V$ simultaneously. Let $U\in \U$ be the unitary matrix exchanging $e_{d(k)}^k$ with $e_1^j$ (the last eigenvector of $N_k$ with the first eigenvector of $N_j$), with $j>k$. Then
$$
\|V\|=\varphi(V)=-\tr(NV)=\sum_{s\in F}v_s \sum_{l=1}^{d(s)}\lambda_l(N_s),
$$
while $-\tr(U^*NUV)=-\tr(NUVU^*)$ is the same sum, but with the two mentioned eigenvalues of $N$ exchanged. By means of (\ref{menor}), after cancelling all the terms that remain equal, we obtain
$$
v_j\lambda_{d(k)}(N_k)+v_k\lambda_1(N_j)\le v_k\lambda_{d(k)}(N_k)+v_j\lambda_1(N_j).
$$
From here we obtain $(v_k-v_j)(\lambda_{d(k)}(N_k)-\lambda_1(N_j))\ge 0$, and since $v_k>v_j$ when $j>k$, the conclusion of the first assertion follows. Note that if $\lambda_{d(k)}(N_k)=\lambda_1(N_j)$ then in fact we have 
$$
\varphi(\tilde{V})=-\tr(NUVU^*)=-\tr(U^*NUV)=-\tr(NV)=\varphi(V)=\|V\|=\|UVU^*\|=\|\tilde{V}\|,
$$
therefore $\tilde{V}\in F_{\varphi}$, proving the second assertion. If this face is a singleton, and the equality  $\lambda_{d(k)}(N_k)=\lambda_1(N_j)$ holds, then $UVU^*=\tilde{V}=V$ but this is impossible since $v_k\ne v_j$. Therefore in that case we must have strict inequality, and this proves the third assertion. Now if we have a permutation of two different eigenprojections of $N_k$, say by means of a unitary $U$, then $\ran(UN_kU^*)=\ran(N_k)=\ran(P_k)$ and $U$ is the identity in the other blocks, hence if $\phi=(UNU^*|\cdot)$ we have
\begin{align*}
\phi(V)& =\sum_{s\ne k} v_s\tr(N_s)+ \sum_l v_l\tr(UN_kU^*P_l)=\sum_{s\ne k} v_s\tr(N_s)+ v_k\tr(UN_kU^*P_k)\\
&=\sum_{s\ne k} v_s\tr(N_s)+ v_k\tr(N_k)=(N|V)=\varphi(V)=\|V\|.
\end{align*}
On the other hand for any $Z\in \lu$ 
$$
\phi(Z)=-\tr(UNU^*Z)=-\tr(NU^*ZU)=(N|UZU^*)\le \|N\|' \|UZU^*\|=\|\varphi\| \|Z\|= \|Z\|,
$$
which combined with the previous equality shows that $\|\phi\|=1$ and $\phi$ is norming for $V$. Finally, if there is exactly one functional norming $V$, then the eigenvalues of each $N_s$ must be equal thus $N_s=n_s P_s$ for each $s$.
\end{proof}

\begin{rem}
Since $N$ commutes with $V$, and any rearranging of the eigenvectors inside each block $N_k$ is a priori admissible, if necessary we may assume that the eigenvalues of $N$ are all ordered decreasingly. On the other hand, it is clear that if $\lambda_{d(k)}(N_k)$ is the smallest eigenvalue of $N_k$, then 
$$
N_k\ge \lambda_{d(k)}(N_k)P_k.
$$
Likewise, if $\lambda_1(N_j)$ is the greatest eigenvalue of $N_j$ then $N_j\le \lambda_1(N_j)P_j$, 
and we recently proved that $\lambda_{d(k)}(N_k)\ge \lambda_1(N_j)$ when $j>k$. Both inequalities here are with respect of the partial order of matrices,
$$
X\ge Y\quad \textrm{ iff }\quad \langle X\xi,\xi\rangle\ge \langle Y\xi,\xi\rangle\quad \forall\; \xi\in \mathbb C^n
$$
\end{rem}

\begin{rem}
If $[X,[X,Y]]=0$ or $[Y,[Y,X]]=0$, then $[X,Y]=0$, because
$$
0=(Y|0)=(Y|[X,[X,Y])=([Y,X]|[X,Y])=-\|[X,Y]\|_F^2,
$$
and then $[X,Y]=0$. For the Frobenius norm (the unique $\Ad$-invariant norm that comes from an inner product), the unique norming functional norming $0\ne V\in \lu$ is given by  $\varphi=(\frac{V}{\|V\|_F}|\;\cdot\;)$, thus $N_V=\frac{V}{\|V\|_F}$, therefore by the previous computation, in the setting of the Frobenius norm
$$
\varphi_V([X,[X,V]])=0
$$
can only happen if $[X,V]=0$. In what follows we are going to explore this condition for other $\Ad$-invariant norms. In particular we will prove that the previous assertion holds true for any strictly convex norm.
\end{rem}

\begin{rem}\label{conmu}If $X=X_D+X_C$ is the decomposition of $X$ in its diagonal and codiagonal parts with respect to $V=i\sum_k v_k P_k$, with $v_1>v_2>\dots >v_F$ as before, then  $[X,V]=[X_C,V]$ since $X_D$ commutes with $V$ (each block of $V$ is a multiple of the identity submatrix).  Then if $\varphi=(N|\cdot)$ norms $V$, then  by  the Jacobi identity and  (\ref{cyclic}) we have
\begin{align*}
\varphi([X,[X,V]])&=\varphi([X,[X_C,V]]) = (N\,|\,[[X,X_C],V] +[X_C,[X,V]])\\
&=-(N\,|\,[V,[X,X_C]]) +(N\,|\, [X_C,[X_C,V]])\\
&=-([N,V]\,|\,[X,X_C]) +([N,X_C]\,|\,[X_C,V])=([N,X_C]\,|\,[X_C,V])
\end{align*}
since $N$ commutes with $V$, that is
\begin{equation}\label{nuc}
\varphi([X,[X,V]])=([N,X_C]\,|\,[X_C,V]).
\end{equation}
\end{rem}

\begin{teo}[The case of equality]\label{teoN}
Let $X,V\in\lu$, write $V=i\sum_k v_k P_k$ with $v_1>\dots>v_F$ and let $\varphi=(N|\,\cdot)$ be any norming functional for $V$. Then $\varphi([X,[X,V]])=0$ if and only if $[X_C,N]=0$. Moreover, if $N$ is diagonal with respect to $V$ (in particular if $\varphi$ is the unique functional norming $V$) then $[X,N]=0$. On the other hand if $F_{\varphi}=\{V\}$, then $[X,V]=0$.
\end{teo}
\begin{proof}
assume that $[X_C,N]=0$, then by (\ref{nuc}) we have $\varphi([X,[X,V]])=([N,X_C]\,|\,[X_C,V])=0$. 
Now assume that $\varphi([X,[X,V]])=0$, first take $Z=P_jXP_k+P_kXP_j\in \lu$, and note that $Z=Z_C$ if $k\ne j$. Then by Propositon \ref{diss} and the identity (\ref{nuc}), we have 
\begin{align}
\tr([Z,N][Z,V]) &=-\tr([N,Z][Z,V])=([N,Z]\,|\,[Z,V])\nonumber \\
&=([N,Z_C]\,|\,[Z_C,V])=\varphi([Z,[Z,V]])\le 0 \label{esta}.
\end{align}
We now compute
\begin{align}
[Z,N]&=i(P_jXP_k+P_kXP_j)\sum_s N_s -i\sum_s N_s (P_jXP_k+P_kXP_j)\\
& =i( P_jXN_k+P_kXN_j -N_jXP_k-N_kXP_j).
\end{align}
Likewise 
$$
[Z,V]=i(P_jX v_kP_k + P_kXv_jP_j-v_jP_j X P_k-v_kP_k XP_j)=i( v_k-v_j)( P_jXP_k-P_kXP_j).
$$
Then
\begin{align*}
\tr([Z,N][Z,V]) &  = -( v_k-v_j)\tr(( P_jXN_k+P_kXN_j -N_jXP_k-N_kXP_j)( P_jXP_k-P_kXP_j))\\
& =-(v_k-v_j)\tr(-P_jXN_kXP_j+P_kXN_jXP_k+N_jXP_kXP_j-N_kXP_jXP_k)\\
& = -2(v_k-v_j)\tr(N_jXP_kX-N_kXP_jX).
\end{align*}
Thus by (\ref{esta}) we have
\begin{equation}\label{menorigual}
2(v_k-v_j)\tr(N_jXP_kX-N_kXP_jX)\ge 0\qquad \forall\; k\ne j.
\end{equation}
We now compute
\begin{align*}
[X_C,N]&=i\sum_{r\ne s}(P_rXP_s+P_sXP_r)\sum_l N_l -i\sum_l N_l \sum_{r\ne s}(P_rXP_s+P_sXP_r)\\
&=i\sum_{r\ne s}(P_rXN_s+P_sXN_r) -i\sum_{r\ne s}(N_rXP_s+N_sXP_r)\\
&=i \sum_{r\ne s}P_rXN_s+P_sXN_r -N_rXP_s-N_sXP_r.
\end{align*}
Likewise,
\begin{align*}
[X_C,V] & =i(\sum_{j\ne j}P_kXv_jP_j+P_jXv_kP_k -v_kP_k XP_j-v_jP_j XP_k)\\
& =i\sum_{k\ne j}(v_k-v_j)(P_jXP_k-P_kXP_j).
\end{align*}
Then $\varphi([X,[X,V]]) =\tr([X_C,N][X_C,V])$ equals
\begin{align*}
& =-\sum_{j\ne k}\sum_{m\ne n,r\ne s}\tr((P_rXN_s+P_sXN_r -N_rXP_s-N_sXP_r)(v_k-v_j)(P_jXP_k-P_kXP_j))\\
& = -4\sum_{k\ne j}(v_k-v_j)\tr(N_jXP_kX-N_kXP_jX).
\end{align*}
By inequality (\ref{menorigual}), this can only be zero if $\tr(N_jXP_kX-N_kXP_jX)=0$ for all $k\ne j$. Assume then that the sum is zero, and consider first the case of $j>k$. Then by Lemma \ref{menorigual}, $\lambda_{d(k)}(N_k)\ge \lambda_1(N_j)$. We are going to show that these numbers are in fact equal and that 
\begin{equation}\label{done}
P_kXP_jN=P_kXN_j=\lambda_1(N_j) P_kXP_j,  \qquad P_jXP_kN=P_j XN_k=\lambda_1(N_j) P_jXP_k.
\end{equation}
Taking adjoints we will also obtain $NP_jXP_k= \lambda_1(N_j) P_jXP_k$, $NP_k XP_j=N_k XP_j=\lambda_1(N_j) P_kXP_j$,  thus for $k>j$ we get 
$$
P_jXP_k N=P_j XN_k=\lambda_1(N_j) P_jXP_k,\quad N P_jXP_k =N_j XP_k=\lambda_1(N_j) P_jXP_k,
$$
i.e. $[P_jXP_k,N]=0$. By symmetry, the same argument will work for $j<k$ (this time with $\lambda_1(N_k)$ instead of $\lambda_1(N_j)$), and then we will have proved the statement of the theorem, $[X_C,N]=0$.

For the purpose of simplifying the  changes of signs in the computation, it will be convenient to write $X=i\widetilde{X}$ with $\widetilde{X}^*=\widetilde{X}$, and in fact we will drop the tilde, abusing a bit the notation (note that  $X_C$ commutes with $N$ if and only if $\widetilde{X}_C$ commutes with $N$). So let us assume that $j>k$, and we start from the data $\tr(N_jXP_kX)=\tr(N_kXP_jX)$ with $X^*=X$ self-adjoint. Then since $N_k\ge \lambda_{d(k)}(N_k)P_k$ and $\lambda_{d(k)}\ge \lambda_1(N_j)$ we get
$$
P_jXN_kXP_j\ge \lambda_{d(k)}(N_k) P_jXP_kXP_j\ge \lambda_1(N_j) P_jXP_kXP_j,
$$
where in the second inequality we used that $P_jXP_kXP_j=\|P_kXP_j\|_F^2\ge 0$. 
Likewise, using $N_j\le \lambda_1(N_j)P_j$ we obtain
$$
P_kXN_jXP_k\le \lambda_1(N_j) P_kXP_jXP_k= \lambda_1(N_j) \|P_jXP_k\|_F^2.
$$
Taking traces and using the hypothesis we obtain
$$
\lambda_1(N_j) \,\|P_kXP_j\|_F^2\le \lambda_{d(k)}(N_k)\,\|P_kXP_j\|_F^2\le \tr(P_jXN_kX)=\tr(P_kXN_jX)\le \lambda_1(N_j)\,\|P_jXP_k\|_F^2.
$$
Since $\|P_jXP_k\|_F^2= \|P_kXP_j\|_F^2=\tr(P_kXP_jX)$ we have in fact the equalities
$$
 \lambda_{d(k)}(N_k)\tr(P_kXP_jX)=\tr(P_jXN_kX)=\tr(P_kXN_jX)=\lambda_1(N_j) \tr(P_kXP_jX)
$$
If $\|P_kXP_j\|_F=0$, the $P_kXP_j=P_jXP_k=0$ and (\ref{done}) is clear. If not, then it must be that $\lambda_{d(k)}(N_k)=\lambda_1(N_j)=\lambda$. Moreover, note that
$$
A=P_jXN_kXP_j- \lambda  P_jXP_kXP_j \ge 0
$$
but we just proved that $\tr(A)=0$. Therefore $A=0$, that is $P_jXN_kXP_j=\lambda  P_jXP_kXP_j$. Likewise,  since  $P_kXN_jXP_k\le \lambda P_kXP_jXP_k$, by the equality of traces obtained it must be that they are equal, therefore we have
\begin{equation}\label{laecu}
P_jXN_kXP_j=\lambda  P_jXP_kXP_j\qquad \textrm{ and} \qquad P_kXN_jXP_k=\lambda P_kXP_jXP_k.
\end{equation}
assume first that $\lambda\ge 0$, let $N_j=N_j^+-N_j^-$ where $N_j^+,N_j^-\ge 0$ is the Hahn decomposition of $N_j$ into positive and negative parts of $N_j$. Since $N_j\le \lambda P_j$ it must be that $0\le N_j^+\le \lambda P_j$ also, and $0\le (N_j^+)^2\le \lambda^2 P_j$. Now
$$
\lambda P_kXP_jXP_k=P_kXN_jXP_k=P_kXN_j^+XP_k-  P_kXN_j^-XP_k\le  \lambda P_kXP_jXP_k-  P_kXN_j^-XP_k,
$$
therefore $0\le P_kXN_j^-XP_k\le 0$ and this is only possible if $P_kXN_j^-XP_k=0$. Since $P_kXN_j^-XP_k=|\sqrt{N_j^-}XP_k|^2$, it follows that it must be $N_j^-XP_k=0=P_kXN_j^-$. Now 
$$
\|P_kXN_j^+\|_F^2=\tr(P_kX(N_j^+)^2XP_k)\le \lambda^2 \tr(P_kXP_jXP_k)
$$
therefore $\|P_kXN_j^+\|_F\le \lambda \|P_jXP_k\|_F$. With this we will obtain an equality in the following Cauchy-Schwarz inequality:
\begin{align*}
\lambda \|P_jXP_k\|_F^2 & =\tr(P_k XN_jXP_k)=\tr(P_k XN_j^+XP_k)=( P_kXN_j^+ | P_kXPj)\\
& \le \| P_kXN_j^+\|_F\| P_kXPj\|_F \le \lambda \|P_kXPj\|_F^2.
\end{align*}
But the Cauchy-Schwarz inequality is an equality if and only if the vectors are aligned \cite{larcs}, i.e.
$P_kXN_j^+ = \mu P_kXPj$,  and multiplying by $P_jXP_k$ on the right we easily see that $\mu=\lambda$. Therefore 
$$
P_kXN_j=P_kXN_j^+-P_kXN_j^-=P_kXN_j^+=\lambda P_kX P_j,
$$
which is the first identity claimed in (\ref{done}). To obtain the second identity, note that $N_k\ge \lambda P_k\ge 0$ thus $\sqrt{N_k}\ge \sqrt{\lambda}P_k\ge 0$ and then $P_j X\sqrt{\lambda} P_kXP_j\le P_jX\sqrt{N_k}XP_j$. On the other hand, using (\ref{laecu}), we have
$$
\|P_jX\sqrt{N_k}\|_F^2=\tr(P_jX\sqrt{N_k}\sqrt{N_k}XP_j)=\tr(P_jXN_kXP_j)=\lambda \|P_kXP_j\|_F^2=\lambda \|P_jXP_k\|_F^2,
$$
and again we obtain an equality in a Cauchy-Schwarz inequality:
\begin{align*}
\sqrt{\lambda}\|P_kXP_j\|_F^2& =\sqrt{\lambda}\tr(P_jXP_kXP_j)\le \tr(P_jX\sqrt{N_k}XP_j)\\
&= (P_jX\sqrt{N_k}|P_jXP_k)\le \|P_jX\sqrt{N_k}\|_F \|P_jXP_k\|_F\le \sqrt{\lambda} \|P_jXP_k\|_F^2.
\end{align*}
In this case we obtain $P_jX\sqrt{N_k}=\sqrt{\lambda} P_jXP_k$, and then
$$
P_j XN_k=P_jX\sqrt{N_k}\sqrt{N_k}=\sqrt{\lambda} P_jXP_k\sqrt{N_k}=\sqrt{\lambda} P_jX\sqrt{N_k}=\lambda P_jXP_k,
$$
which is the second identity claimed in (\ref{done}). This settles the case of $\lambda\ge 0$. The case of $\lambda<0$ can be dealt in a similar fashion, with some modifications: start by noting that $N_j\le \lambda P_j<0$, thus $-N_j\ge 0$ and moreover $\sqrt{-N_j}\ge \sqrt{-\lambda}P_j\ge 0$. Hence $P_kX\sqrt{-N_j}XP_k\ge P_kXP_j\sqrt{-\lambda}XP_k$ and
\begin{align*}
\sqrt{-\lambda}\|P_jXP_k\|_F^2 &\le \tr(P_kX\sqrt{-N_j}XP_k)=(P_kX\sqrt{-N_j}|P_kXP_j)\\
& \le \|P_kX\sqrt{-N_j}\|_F \|P_kXP_j\|_F\le \sqrt{\tr(P_kX(-N_j)XP_k)} \|P_kXP_j\|_F\\
&=\sqrt{-\lambda\tr(P_kXP_jXP_k)}\|P_kXP_j\|_F=\sqrt{-\lambda}\|P_jXP_k\|_F^2.
\end{align*}
Once more we have an equality in the Cauchy-Schwarz inequality for the trace inner product, which is only possible if $P_kX\sqrt{-N_j}=\sqrt{-\lambda}P_kXP_j$. Argumenting as before (multipliying by $\sqrt{-N_j}$ on the right) we arrive to $P_kXN_j=\lambda P_kXP_j$, and this establishes the first identity in (\ref{done}).  Now note that since $N_k\ge \lambda P_k$ with $\lambda<0$, then if $N_k=N_k^+-N_k^-$ is the Hahn decomposition of $N_k$, it must be that $-N_k^-\ge \lambda P_k$ and $(N_k^-)^2=(-N_k^-)^2\le \lambda^2 P_k$. Argumenting like in the previous case we first see that $P_jXN_k^+XP_j=0$, i.e. $P_jXN_k^+=N_k^+XP_j=0$. Thus $P_jXN_k=P_jXN_k^-$ and
$$
\|P_jXN_k\|_F^2=\tr(P_jXN_k^2XP_j)=\tr(P_jX(N_k^-)^2XP_j)\le \lambda^2 \tr(P_jXP_kXP_j)=\lambda^2 \|P_kXP_j\|_F^2.
$$
For the last time we deal with a Cauchy-Schwarz inequality
\begin{align*}
-\lambda \|P_kXP_j\|_F^2&  =\tr(-P_jXN_kXP_j)=(-P_jXN_k|P_jXP_k)\le \|P_jXN_k\|_F \|P_kXP_j\|_F\\
 & \le |\lambda|\, \|P_kXP_j\|_F^2= -\lambda \|P_kXP_j\|_F^2,
\end{align*}
which now gives us $P_jXN_k=\lambda P_jXP_k$, the second identity in (\ref{done}). Now that we know that $[X_C,N]=0$, recall form Proposition  \ref{ordenados} that when there exists only one $\varphi$ norming $V$, then $N$ is diagonal, i.e. $N=i\sum_k n_kP_k$. But the diagonal part of $X$ (with respect to $V$) is then also diagonal with respect to $N$, or equivalently, $[X_D,N]=0$. Hence in this case $[X,N]=[X_C,N]+[X_D,N]=0+0=0$. On the other hand, if $F_{\varphi}=\{V\}$, then by the same proposition we can't have the equalities $\lambda_{d(k)}(N_k)=\lambda_1(N_j)$, so at each step it must be that $P_kXP_j=0$, showing that $X_C=0$, hence $X=X_D$ or equivalently $X$ commutes with $V$.
\end{proof}

\begin{rem}[Strictly convex norms]\label{scn} If the norm is strictly convex, then each face of the sphere is a singleton, thus in particular $\varphi([X,[X,Y]])=0$ implies $[X,Y]=0$ if $\varphi$ supports $X$.
\end{rem}

\begin{rem}
If $\varphi([X,[X,V]])=0$ for some non-zero $X,V$, and $\varphi$ norming $V$, we can write $V$ as a convex combination of extreme points of the sphere of radius $\|V\|$, i.e. $V=\sum_i\lambda_i V_i$, with $0<\lambda_i\le 1$ and $\sum_i\lambda_i=1$. Since the $V_i$ must all be in the same face as $V$, the functional $\varphi$ norms each of the $V_i$. Then we have 
$$
0=\varphi([X,[X,V]])=\sum_i \lambda_i \varphi([X,[X,V_i]]).
$$
it must be that $\varphi([X,[X,V_i]])=0$ for each $V_i$ (all the terms are non-positive by Lemma \ref{diss}). With a similar reasoning,  let $\varphi=(N|\cdot)=\sum_i\varphi_i=\sum_i(N^i|\cdot)$ with $\varphi_i$ extremes, then it must be $\varphi_i([X,[X,V])=0$.
\end{rem}

\begin{lem}\label{N0}
Let $V=i\sum_k v_kP_k\in\lu$ with $v_1>v_2>\dots>v_F$. There exists $\varphi_0=(N_0|\cdot)$ norming $V$ such that $N_0$ is diagonal, i.e. $N_0=i\sum_{k\in F}\lambda_k P_k$. Let $X\in \lu$, then $\varphi_0([X,[X,V]])=0$ if and only if $[X,N_0]=0$.
\end{lem}
\begin{proof}
Let $\varphi=(N|\,\cdot)$ be any norming functional for $V$, then $N=i\sum_{k\in F}N_k$ with $N_k^*=N_k=N_kP_k=P_k$ for all $k\in F$. As before, choose an orthonormal basis $\{e_1^k,\dots,e_{d(k)}^k\}$ of $\ran(P_k)$ diagonalizing each $N_k$ (here $d(k)=dim(P_k)$) i.e. we have 
$$
N_k e_l^k= \lambda_l(N_k)e_l^k\qquad l=1,\dots,d(k)
$$
for each $k\in F$; putting these basis together in the given order, we obtain a basis that diagonalizes $N$ and $V$ simultaneously. Now for each $k\in F$ and for each $l,s\in F_k=\{1,\dots,d(k)\}$, let $U_{k,ls}$ be the unitary matrix exchanging $e_l^k$ with $e_s^k$. Clearly 
$$
U_{k,ls}P_jU_{k,ls}^*=P_j\qquad \forall\; j,k\in F,
$$
thus $U_{k,ls}V U_{k,ls}^*=V$. Note that for each $k$, there are exactly $C(k)=\binom{d(k)}{2}$ ways to pick $l\ne s$ from each $F_k$, and let us denote $|F|$ to the cardinal number of $F$. Let
$$
N_0=\frac{1}{|F|}\sum_{k\in F}\frac{1}{C(k)}\sum_{l\ne s\in F_k}U_{k,ls}NU_{k,ls}^*.
$$
We claim that this $N_0$ has all the desired properties. First, note that by shuffling the diagonal entries from each $N_k$, each new block is a diagonal block with the arithmetic mean of the eigenvalues of 
$N_k$ in the diagonal, that is $N_0=i\sum_k\lambda_kP_k$ with 
$$
\lambda_k=\frac{1}{d(k)}\sum_{l=1}^{d(k)}\lambda_l(N_k).
$$
Now, note that
$$
\|N_0\|'\le \frac{1}{|F|}\sum_{k\in F}\frac{1}{C(k)}\sum_{l\ne s\in F_k}\|N\|'=\frac{1}{|F|}|F|\frac{1}{C(k)}C(k)=1
$$ 
and that
\begin{align*}
(N_0|V)=\frac{1}{|F|}\sum_{k\in F}\frac{1}{C(k)}\sum_{l\ne s\in F_k}(N|U_{k,ls}^*VU_{k,ls})=\frac{1}{|F|}\sum_{k\in F}\frac{1}{C(k)}\sum_{l\ne s\in F_k}\|V\|=\|V\|,
\end{align*}
therefore $\|N_0\|'=1$ and $\varphi_0=(N_0|\cdot)$ norms $V$. If $X\in \lu$ we have 
$$
\varphi_0([X,[X,V]])=(N_0|[X,[X,V]])=([N_0,X]|[X,V])
$$
from where it is clear that $[X,N_0]=0$ implies that this number is $0$. On the other hand and by Theorem \ref{teoN}, if this number is $0$ then $[X_C,N]=0$, but since each block of $N_0$ is diagonal we also have $[X_D,N]=0$.
\end{proof}

\begin{rem}Let $N$ norm $X$; since the eigenvalues of $N_k$ are greater or equal than those of $N_{k+1}$, so is the arithmetic mean of those. Hence using the notation of the previous lemma we have $\lambda_k\ge \lambda_{k+1}$, with equality $\lambda=\lambda_k=\lambda_{k+1}$ if and only if all of them are equal. That is if and only if $\lambda_l(N_k)=\lambda_s(N_{k+1})=\lambda$ for any $l\in\{1,\dots,d(k)\}, s\in \{1,\dots,d(k+1)\}$. This tells us that it must be $N_k+N_{k+1}=\lambda (P_k+P_{k+1})$. On the other hand, if that is not the case, then $\lambda_k>\lambda_{k+1}$, and going through the proof of Theorem \ref{teoN}, in this case $\varphi_0([X,[X,V]])=0$  implies $P_kXP_{k+1}=0$.
\end{rem}

\section{The adjoint action of $\lu$ on itself}\label{action}
 
In this section we give an application of Theorem \ref{teoN}. We begin by giving a characterization of the lateral directional derivatives of a (Finsler) norm function.

\subsection{Lateral derivatives} 

\begin{rem}\label{cocientesincrementales} Let $h:\mathbb R\to \mathbb R$ be a continuous convex function then
$\lim\limits_{t\to 0^+} \frac{h(x+t)-h(x)}{t}$ exists (in the sense that it is a real number or $-\infty$), moreover the quotient is decreasing as $t\to 0^+$: from the convexity follows that if $0<s<t$ then
\begin{equation}\label{decrece}
\frac{h(x+s)-h(x)}{s}\leq \frac{h(x+t)-h(t)}{t},
\end{equation}
showing the monotonicity in the variable $t\to 0^+$, and from here the claim on the limit is apparent. Now for given $X,Y$ in a Finsler normed space $(E,\|\cdot\|)$, with $X\ne 0$, let $h(t)=\|X+tY\|$, and note that this is a continuous convex function. Therefore the limit
$$
\lim\limits_{t\to 0^+} \frac{\|X+tY\|-\|X\|}{t}
$$
exists; due to the fact that the quotient $\frac{\|X+tY\|-\|X\|}{t}$ is bounded below by the number $-\|Y\|$, it follows that the limit is a finite number. Exchanging $t$ with $-t$ we see that $t^{-1}(\|X+tY\|-\|X\|)$ is increasing as $t\to 0^-$, and that the limit also exists; moreover for any $t>0$ we have that
$$
2\|X\|=\|X+tY+X-tY\|\leq \|X+tY\|+\|X-tY\|
$$
which implies that $\|X\|-\|X+tY\|\leq \|X-tY\|-\|X\| $. Dividing both sides by $-t$, it follows that
\begin{equation}\label{rl}
\lim\limits_{t\to 0^+}\frac{\|X+tY\|-\|X\|}{t}\geq \lim\limits_{t\to 0^+}\frac{\|X-tY\|-\|X\|}{-t}=\lim\limits_{t\to 0^-}\frac{\|X+tY\|-\|X\|}{t}.
\end{equation}
\end{rem}

\begin{defi}[Norming functionals]\label{normif} It will be convenient to denote, for $0\ne X\in\lu$
$$
N_X=\{\varphi\in \lu': \|\varphi\|= 1 \hspace{0.2cm} \text{ and }\hspace{0.2cm} \varphi(X)=\|X\|\}.
$$
As we already discussed extensively, this is a nonempty compact convex set in the dual space. Moreover, if $\varphi\in N_X$ then also $\varphi\in N_{\lambda X}$ for any $\lambda>0$ (see Definition \ref{caras}). 
\end{defi}

The following result is well-known for norms (i.e. fully homogeneous). Since we could not find a suitable reference for Finsler norms, we include a proof:

\begin{prop}\label{normderivative} Let $(E,\|\cdot\|)$ be a Finsler normed space and $0\ne X\in E$, then for any $Y\in E$
$$
\lim\limits_{t\to 0^+}\frac{\|X+tY\|-\|X\|}{t}=\max\limits_{\varphi\in N_X}\varphi(Y).
$$
\end{prop}
\begin{proof}
It will be convenient to use the auxiliar map $h(t)=\|X+tY\|$ and the notation $h'(0^+)$ for the limit we are computing. We will also use $h'(0^-)$ for the limit with $t\to 0^-$. First note that for each $\varphi\in N_X$, we have
$$
\frac{h(t)-h(0)}{t}=\frac{\|X+tY\|-\|X\|}{t}\ge \frac{\varphi(X+tY)-\varphi(X)}{t}=\varphi(Y).
$$
Thus $h'(0^+)\ge \max\limits_{\varphi\in N_X}\varphi(Y)$. Let us show that there exists at least one $\varphi\in N_X$ where the equality holds, which will finish the proof. Let $\psi:\text{span}\{x,y\}\longrightarrow \mathbb R$  be the linear functional given by $\psi(aX+bY)=a\|X\|+b h'(0^+)$, for $a,b\in\mathbb R$. Note that $\psi(X)=\|X\|$ and $\psi(Y)=h'(0^+)$. We claim that  $\psi(aX+bY)\leq \|aX+bY\|$, this is clear if $a=0$ or $b=0$. Now recall that $h$ is a decreasing function by (\ref{decrece}). First, suppose that $a>0$ and $b>0$, then
\begin{align*}
a\|X\|+b h'(0^+)&=a\|X\|+b\lim\limits_{t\to 0^+}\frac{\|X+tY\|-\|X\|}{t}\leq a\|X\|+b\frac{\|X+\frac{b}{a}Y\|-\|X\|}{\frac{b}{a}}\\
&=a\|X\|+a\|X+\frac{b}{a}Y\|-a\|X\|=\|aX+bY\|.
\end{align*}
If $a<0$ and $b<0$, recall that the limit from the right is less or equal than the limit from left (\ref{rl}), and that the limit from the left is increasing; since $b<0$ we see that 
\begin{align*}
a\|X\|+b  h'(0^+) &\leq a\|X\|+b  h'(0^-)=a\|X\|+b\lim\limits_{t\to 0^-}\frac{\|X+tY\|-\|X\|}{t}\leq a\|X\|+b\frac{\|X-Y\|-\|X\|}{-1}\\
&=a\|X\|+b\|Y\| -b\|X-Y\|=\|bY-bX\|-\|-(a+b)X\|\\
&\leq \|bY-bX+(a+b)X\|=\|aX+bY\|.
\end{align*}
If $a<0$ and $b>0$ then
\begin{align*}
a\|X\|+b  h'(0^+)&=a\|X\|+b\lim\limits_{t\to 0^+}\frac{\|X+tY\|-\|X\|}{t}\leq a\|X\|+b\frac{\|X-\frac{b}{a}Y\|-\|X\|}{-\frac{b}{a}}\\
&=2a\|X\|+\|-aX+bY\|=\|-aX+bY\|-\|-2aX\|\\
&\leq \|-aX+bY+2aX\|=\|aX+bY\|.
\end{align*}
If $a>0$ and $b<0$ then
\begin{align*}
a\|X\|+b  h'(0^+)& \leq a\|X\|+b \frac{\|X-\frac{b}{a}Y\|-\|X\|}{-\frac{b}{a}}=2a\|X\|-a\|X-\frac{b}{a} Y\|\\
&=\|2aX\|-\|aX-bY\|\leq \|2aX-aX+bY\|=\|aX+bY\|.
\end{align*}
Therefore  $\psi(aX+bY)\leq \|aX+bY\|$ for all $a,b\in\mathbb R$, thus $\|\psi\|=1$. By the Hahn-Banach theorem there exists $\varphi\in E'$ that extends $\psi$ with the same norm; in particular $\varphi(X)=\|X\|$ thus $\varphi\in N_X$. But also $\varphi(Y)=h'(0^+)$, and this finishes the proof.
\end{proof}

\begin{rem} By the previous discussion $h:t\mapsto t^{-1}(\|X+tY\|-\|X\|)$ is non-increasing for $t>0$, and exchanging $Y$ with $-Y$ we see that for $t<0$ the function $h$ is non-decreasing, for $s<0<t$ we have
$$
\frac{\|X+sY\|-\|X\|}{s}\le \frac{\|X+tY\|-\|X\|}{t},
$$
and for each $X\ne 0$ we have
\begin{equation}\label{derlat}
\lim\limits_{t\to 0^-}\frac{\|X+tY\|-\|X\|}{t}=\min\limits_{\varphi\in N_X}\varphi(Y)\le \max\limits_{\varphi\in N_X}\varphi(Y)=\lim\limits_{t\to 0^+}\frac{\|X+tY\|-\|X\|}{t}.
\end{equation}
\end{rem}

\subsection{Dissipative operators and inequalities}

We now recall that an operator $T$ in a complex Banach space $(E,\|\cdot\|_E$) is \textit{dissipative} if for each $\xi\in E$ and each norming functional $\phi\in E^*$ of the vector $\xi$, we have that $\mathrm{Re}\phi(T\xi)\le 0$. From Hille-Yosida's theorem and the theory of dissipative operators (see  \cite[Chapter 1]{pazy} and \cite{lineq}) it follows that for any $T\in \mathcal {B}(E)$

\begin{teo}The following are equivalent
\begin{enumerate}
\item For each $\xi\in E$ there exists some norming functional $\phi\in E^*$ of $\xi$, such that $\mathrm{Re}\phi(T\xi)\le 0$.
\item $\|e^{sT}\|\le 1$ for all $s\ge 0$.
\item $1-sT$ is invertible and expansive for each $s\ge 0$, i.e. $\|\xi -sT\xi\|_E\ge \|\xi\|_E$.
\end{enumerate}
\end{teo}

\begin{rem}[Complexification and the Taylor norm]\label{taylor} Let $A,B\in \lu$, and for a given $\Ad$-invariant norm in $\lu$, let 
$$
\|A+iB\|_T=\sup_{t\in [0,2\pi]}\|A\cos t-B\sin t\|
$$ 
be the Taylor norm of  $A+iB$ in the complexification $M_n(\mathbb C)=\lu\oplus i\lu$ of $\lu$. This is a norm in the complexification that extends the norm in $\lu$, and it is easy to check that $\|A+iB\|_T=\|A-iB\|_T$ and that $\|U(A+iB)U^*\|_T=\|A+iB\|_T$ for any unitary $U$ (in fact, there are may possible complexifications, see \cite{munioz})). Side question: can we recover the $p$-norm on the full $M_n(\mathbb C)$ using any of the standard complexification procedures listed in \cite{munioz}, applied to the $p$-norm in $\lu$?
\end{rem}

\begin{teo}For any $X\in \lu$, the operators $\ad X$ and $\ad^2X$ are dissipative. In particular $1+s\ad X$ is invertible and expansive for all $s\in\mathbb R$, and $1-s\ad^2 X$ is invertible and expansive for all $s\ge 0$.
\end{teo}
\begin{proof}
If we complexify $\ad X$, i.e. $\ad X(A+iB)=\ad X(A)+i \ad X(B)$, then it follows that  $\|e^{s\ad X}(A+iB)\|_T=\|e^{sX}(A+iB)e^{-sX}\|_T=\|A+iB\|_T$. Thus the complexification of $\ad X$ is dissipative, which implies that $\ad X$ is also dissipative. Changing $X$ with $sX$ and $-sX$, the same holds true for $\pm s\ad X$. Now we can write $1-s^2 \ad^2 X=(1-s\ad X)(1+s\ad X)$, and since both factors are invertible and expansive, so is the product, hence $\ad^2X$ is also dissipative.
\end{proof}

\begin{teo}\label{vmasxv} Let $X,V\in\lu$ and let $\|\cdot\|$ be an $\Ad$-invariant Finsler norm in $\lu$. Then  if $0\le s\le s'$ we have that
$$
\|V\|\le \|V-s[X,[X,V]]\|\le \|V-s'[X,[X,V]]\|
$$
and if  $\,\nicefrac{t'}{t}\ge 1$ we have
$$
\|V\|\le \|V+t[X,V]\|\le \|V+t'[X,V]\|.
$$
\end{teo}
\begin{proof}
We write $\ad^2X=(\ad X)^2$ for short. Let $s\ge 0$, then $1-s\ad^2X$ is expansive and invertible and in particular $\|(1-s\ad^2 X)^{-1}\|\le 1$. Let $\xi=A+iB\in (\lu\oplus i\lu,\|\cdot \|_T)$, let $\phi$ be any norming functional of $\xi$, then
$$
\mathrm{Re}\phi( ((1-s\ad^2 X)^{-1}-1)\xi)=\mathrm{Re}\phi( ((1-s\ad^2 X)^{-1}\xi)-\|\xi\|_T\le \|\xi\|_T-\|\xi\|_T=0,
$$
thus $(1-s\ad^2 X)^{-1}-1=s\ad^2X (1-s\ad^2 X)^{-1}$ is dissipative, and so is $\ad^2X (1-s\ad^2 X)^{-1}$, since any positive multiple of a dissipative operator is dissipative. Thus if $0\le s\le s'$, the operator 
$$
A=(1-s'\ad^2X)(1-s\ad^2X)^{-1}=1-(s'-s)\ad^2 X(1-s\ad^2X)^{-1}
$$
is expansive and invertible, and therefore its inverse is a contraction. Hence
\begin{align*}
\|(1-s\ad^2X)V\|& =(1-s\ad^2X)(1-s'\ad^2X)^{-1}(1-s'\ad^2X)V\|\\
&=\|A^{-1}(1-s'\ad^2X)V\|\le \|(1-s'\ad^2X)V\|,
\end{align*}
and this proves the first assertion. The second assertion has a very similar proof, which is therefore omitted.
\end{proof}

\begin{rem}[Majorization] Using Proposition \ref{majo} we can restate the previous theorem in terms of majorization: if $\lambda(A)\subset \mathbb R^n$ denotes the string of eigenvalues of $A=\sum_k i\lambda_k p_k\in \lu$, then for any $V,X\in \lu$ and $s'/s\ge 1$ we have
$$
\lambda(V)\prec \lambda( V+s[X,V])\prec \lambda(V+s'[X,V]),
$$
and a similar statement for $-[X,[X,V]]$.
\end{rem}

\begin{rem}[Birkhoff orthogonality]\label{bhort}
The previous result can also be rephrased as follows: for any $X,V\in\lu$ and any $\Ad$-invariant norm in $\lu$, the vector $V$ is Birkhoff orthogonal to the subspace $S$ spanned by $[X,V]$ i.e. 
$$
\|V\|=\inf\limits_{s\in\mathbb R}\|V-s[X,V]\|=\mathrm{dist}(V,S).
$$
\end{rem}

\subsection{The case of equality}

For this particular case, and with the equality criteria we developed (Theorem \ref{teoN}), we now add to the equivalences of Proposition \ref{majo} a new equivalence regarding norming functionals and the $ad$-action, for the case of equality in Theorem  \ref{vmasxv} above:

\begin{teo}\label{conotang}Let $X,V\in \lu$ and let $\|\cdot\|$ be an $\Ad$-invariant Finsler norm in $\lu$. The following are equivalent:
\begin{enumerate}
\item $\|V+[X,V]\|=\|V\|$
\item There exists a norming functional  $\psi=(N|\cdot)$ of the matrix $V+[X,V]$ such that $[N,X_C]=0$.
\item $V$ and $V+[X,V]$ belong to a same face of the sphere.
\item $V+[X,V]$ belongs to all the faces of the sphere where $V$ sits in.
\end{enumerate}
The same assertions hold if we replace $[X,V]$ with $-[X,[X,V]]$. If $V$ is a smooth point of the norm, then in the second condition we can replace $[N,X_C]=0$ with $[N,X]=0$. If there exists $\varphi$ norming $V$ such that $F_{\varphi}=\{V\}$ (in particular, if the norm is strictly convex) then equality can only occur if $[X,V]=0$.
\end{teo}
\begin{proof}
Assume that the first condition holds. By the previous theorem, we have for $0<h<1$ that
$$
\|V\|=\|V+[X,V]\|\ge \|V+(1-h)[X,V]\|\ge \|V\|.
$$
Therefore if $f(s)=\|V+s[X,V]\|$, we have that 
$$
0=f'(1^-)=\lim_{h\to 0^+}\frac{\|V+(1-h)[X,V]\|-\|V+[X,V]\|}{-h}=\min\limits_{\psi\in N_{V+[X,V]}}([X,V]).
$$
by equation (\ref{derlat}). Therefore there exists $\psi=(N|\cdot)$ norming $V+[X,V]$ such that $\psi([X,V])=0$. Now note that 
$$
\|V\|=\|V+[X,V]\|=(N|V+[X,V])=(N|V)
$$
therefore $\psi$ norms $V$ also. On the other hand, since $\psi$ norms $V+[X,V]$, we have that $\psi\circ \ad V+\psi\circ \ad [X,V]=0$ by Lemma \ref{diss}. Thus
$$
\psi([X,[X,V]])=-\psi\circ \ad[X,V](X)=\psi\circ\ad V(X)=\psi([V,X])=0,
$$
and by Theorem \ref{teoN} it must be that $[N,X_C]=0$, and this shows $1.\Rightarrow 2.$ From the same theorem we see that it must be $[N,X]=0$ if the norm is smooth at $V$, and also that it must be $[X,V]=0$ if the norm is strictly convex. Now assume that the second assertion holds, we see that 
\begin{align*}
\|V\|&\ge \psi(V)=(N|V)=(N|V)+([N,X_C]|V)=(N|V)+(N|[X_C,V])\\
&=(N|V+[X,V])=\psi(V+[X,V])=\|V+[X,V]\|\ge \|V\|
\end{align*}
thus $\psi$ norms simultaneously $V$ and $V+[X,V]$ and they have the same norm, so $2.\Rightarrow 3.$. If the third assertion holds, then it is plain that $\|V+[X,V]\|=\|V\|$, therefore if $V\in F_{\varphi}$ we have that $\varphi([X,V])=0$ hence
$$
\|V\|=\varphi(V)=\varphi(V+[X,V])\le \|V+[X,V]\|=\|V\|
$$
and then $V+[X,V]\in F_{\varphi}$, thus $3.\Rightarrow 4.$ That the fourth assertion implies the first one is apparent. The proof for $V-[X,[X,V]]$ is similar and therefore omitted.
\end{proof}

\begin{coro} Let $X,V\in\lu$, let $\lambda\in\mathbb R$. Then
\begin{enumerate}
\item If $\lambda_k(V)=\pm i\lambda$ for all $k$, then $\|V+[X,V]\|_{\infty}>\|V\|_{\infty}$ unless $[X,V]=0$.
\item If $V=i\lambda P$  with $P$ a one-dimensional projection, then $\|V+[X,V]\|_1>\|V\|_1$ unless $[X,V]=0$.
\item Let $1<k<n$, let $V$ be as in any of the previous two assertions. Then $\|X+[X,V]\|_{(k)}>\|V\|_{(k)}$ unless $[X,V]=0$.
\end{enumerate}
Similar statements hold replacing $[X,V]$ with $-[X,[X,V]]$.
\end{coro}
\begin{proof}
In the first case, $V$ is an extreme point of the unit sphere and therefore it can be exposed with some $\varphi$ such that $F_{\varphi}=\{V\}$, hence the conclusion follows from the previous theorem. The same proof works for the second assertion, since those are the extreme points of the unit sphere of the trace norm. For the third assertion, one uses a similar argument and the characterization of extreme points of the Ky-Fan norms given in \cite[Theorem 4]{grone}.
\end{proof}

\begin{problem} The conditions of Theorem \ref{conotang} above imply that $[N_V,X_C]=0$ for any norming functional $\varphi_V=(N_V|\cdot)$ of $V$. This is because if $4.$ holds and $\varphi=(N|\cdot)$ is any norming functional for $V$, then it is also norming for $V+[X,V]$, hence repeating the argument of $1.\Rightarrow 2.$ we have that $[N,X_C]=0$.  Does $[N,X_C]=0$ for all $N$ norming $V$ imply that $\|V+s[X,V]\|=\|V\|$ in some interval  $|s|<\delta$?
\end{problem}

\begin{rem}With the adequate precautions, the results of this paper can be stated in the setting of Hermitian matrices.
\end{rem}

\subsection*{Acknowledgements} This research was supported by Consejo Nacional de Investigaciones Cient\'\i ficas y T\'ecnicas (CONICET), Agencia Nacional de Promoci\'on de Cienca y Tecnolog\'\i a (ANPCyT), and Universidad de Buenos Aires (UBA), Argentina.





\begin{thebibliography}{XX}

\bibitem{absil} P.-A. Absil, R. Mahony, R. Sepulchre: {\sc Optimization Algorithms on Matrix Manifolds}. Princeton University Press, Princeton (2008).

\bibitem{ando} T. Ando: \textit{Majorization, doubly stochastic matrices, and comparison of eigenvalues}. Linear Algebra Appl. 118 (1989), 163--248.


\bibitem{alr} E. Andruchow, G. Larotonda, L. Recht: \textit{Finsler geometry and actions of the $p$-Schatten unitary groups}. Tran. Amer. Math. Soc. 362 (2010) no. 1, 319--344.

\bibitem{cocoeste} E. Andruchow, L. Recht: \textit{Sectional curvature and commutation of pairs of selfadjoint operators}. J. Operator Theory 55 (2006) no. 2, 225--238.

\bibitem{alv} J. Antezana, G. Larotonda, A. Varela: \textit{Optimal paths for symmetric actions in the unitary group}. Comm. Math. Phys. 328 (2014), no. 2, 481--497.

\bibitem{atkin1} C. J. Atkin: \textit{The Finsler geometry of groups of isometries of Hilbert space}. J. Austral. Math. Soc. Ser. A 42 (1987) no. 2, 196--222.

\bibitem{befi} J. J. Benedetto, M. Fickus: \textit{Finite Normalized Tight Frames}. Advances in Computational Mathematics 18, 357--385 (2003).

\bibitem{bhatia} R. Bhatia: {\sc Matrix analysis}. Graduate Texts in Mathematics 169. Springer-Verlag, New York, 1997.


\bibitem{bgh} L. Biliotti, A. Ghigi, P. Heinzner. \textit{Coadjoint orbitopes}. Osaka J. Math. 51 (2014), no. 4, 935--968.

\bibitem{dmlr0} C. E. Dur\'an, L. E. Mata-Lorenzo, L. Recht: \textit{Natural variational problems in the Grassmann manifold of a $C^*$-algebra with trace}. Adv. Math. 154 (2000), no. 1, 196--228.

\bibitem{fabian} M. Fabian, P. Habala, P. H\'ajek, V. Montesinos Santaluc\'\i a, J. Pelant, V. Zizler: {\sc Functional analysis and infinite-dimensional geometry}. CMS Books in Mathematics/Ouvrages de Math\'ematiques de la SMC, 8. Springer-Verlag, New York, 2001.

\bibitem{fimi} M. Fickus, D. G. Mixon, J.C. Tremain: \textit{Steiner equiangular tight frames}. Linear Algebra and its Applications 436 (2012) no.5, 1014--1027.

\bibitem{grone} R. Grone, M. Marcus: \textit{Isometries of matrix algebras}. J. Algebra 47 (1977), no. 1, 180--189

\bibitem{jensen} A. Jensen: \textit{Self-polar polytopes}. Polytopes and discrete geometry, 101--124, Contemp. Math., 764, Amer. Math. Soc., Providence RI, 2021.


\bibitem{fan} K. Fan, I. Glicksberg: Some geometric properties of the spheres in a normed linear space. Duke Math. J. 25 (1958), 553--568.


\bibitem{ela} N. Johnston, S. Moein, R. Pereira, S. Plosker: \textit{Birkhoff--James Orthogonality in the Trace Norm, with Applications to Quantum Resource Theories}. Electronic Journal of Linear Algebra 38 (2022) 760--776.

\bibitem{haya} M. Hayashi: {\sc Quantum information theory}. Mathematical foundation. Second edition. Graduate Texts in Physics. Springer-Verlag, Berlin, 2017.

\bibitem{hema} R. Hemasinha: \textit{Permutation bases}. International Journal of Mathematical Education in Science and Technology 25 (1994), no.1, 103--111. 

\bibitem{hole} A. S. Holevo: \textit{Bounds for generalized uncertainty of the shift parameter}. Probability theory and mathematical statistics (Tbilisi, 1982), 243--251, Lecture Notes in Math., 1021, Springer, Berlin, 1983.

\bibitem{horn} A. Horn: \textit{Doubly stochastic matrices and the diagonal of a rotation matrix}. Amer. J. Math. 76 (1954), 620--630.

\bibitem{lineq} G. Larotonda: \textit{Norm inequalities in operator ideals}. J. Funct. Anal. 255 (2008), no. 11, 3208--3228.


\bibitem{larcs} G. Larotonda: \textit{The case of equality in H\"older's inequality for matrices and operators}. Math. Proc. R. Ir. Acad. 118A (2018), no. 1, 1--4. 

\bibitem{lar19} G. Larotonda: \textit{Metric geometry of infinite-dimensional Lie groups and their homogeneous spaces}. Forum Math. 31 (2019), no. 6, 1567--1605.

\bibitem{larmi} G. Larotonda, M. Miglioli: \textit{Hofer's metric for compact Lie groups}, arXiv prerpint (2020).

\bibitem{munioz} G. Mu\~{n}oz, Y. Sarantopoulos, A. Tonge: \textit{Complexifications of real Banach spaces, polynomials and multilinear maps}. Studia Math. 134 (1999), no. 1, 1--33.

\bibitem{nielsen} M. A. Nielsen, I. L. Chuang: {\sc Quantum computation and quantum information}. Cambridge University Press, Cambridge, 2000.

\bibitem{pazy} A. Pazy: {\sc Semigroups of linear operators and applications to partial differential equations}. Applied Mathematical Sciences, 44. Springer-Verlag, New York, 1983.

\bibitem{tropp} J. A. Tropp: \textit{Just relax: convex programming methods for identifying sparse signals in noise}. IEEE Transactions on Information Theory, vol. 52, no. 3, pp. 1030--1051.

\bibitem{watson1} G. A. Watson: \textit{Linear best approximation using a class of polyhedral norms}. Numer. Algorithms 2 (1992), no. 3-4, 321--335.

\bibitem{watson} G. A. Watson: \textit{On matrix approximation problems with Ky Fan $k$ norms}. Algorithms for approximation, III (Oxford, 1992). Numer. Algorithms 5 (1993), no. 1-4, 263--272.

\bibitem{zietak} K. Zi\k{e}tak: \textit{On the characterization of the extremal points of the unit sphere of matrices}. Linear Algebra Appl. 106 (1988), 57--75.


\end{thebibliography}
\end{document}